\documentclass[11pt]{article}
\usepackage{amscd}
\usepackage{amsfonts}
\usepackage{amsmath}
\usepackage{amssymb}
\usepackage{amsthm}
\usepackage{bbm}
\usepackage{CJK}
\usepackage{fancyhdr}
\usepackage{graphicx}
\usepackage{hyperref}
\usepackage{indentfirst}
\usepackage{latexsym}
\usepackage{mathrsfs}
\usepackage{xypic}
\usepackage[compress]{cite}

\newtheorem{theorem}{Theorem}[section]
\newtheorem{lemma}[theorem]{Lemma}
\newtheorem{definition}[theorem]{Definition}
\newtheorem{proposition}[theorem]{Proposition}
\newtheorem{example}[theorem]{Example}

\newtheorem{cor}[theorem]{Corollary}
\newtheorem{remark}[theorem]{Remark}

\usepackage[top=1in,bottom=1in,left=1.25in,right=1.25in]{geometry}
\def\<{\langle}
\def\>{\rangle}
\def\a{\alpha}
\def\b{\beta}

\def\c{\cdot}

\date{}
\begin{document}
\renewcommand{\baselinestretch}{1.2}
\renewcommand{\arraystretch}{1.0}
\title{\bf On $\alpha$-type (equivariant) cohomology of Hom-pre-Lie algebras}
\date{}
\author{{\bf Shuangjian Guo$^{1}$, Ripan Saha$^{2}$\footnote
        { Corresponding author:~~ripanjumaths@gmail.com} }\\
{\small 1. School of Mathematics and Statistics, Guizhou University of Finance and Economics} \\
{\small  Guiyang  550025, P. R. of China} \\
{\small 2. Department of Mathematics, Raiganj University } \\
{\small  Raiganj, 733134, West Bengal, India}}
 \maketitle
\begin{center}
\begin{minipage}{13.cm}

{\bf \begin{center} ABSTRACT \end{center}}
In this paper, we define a new cohomology theory for multiplicative Hom-pre-Lie algebras which controls deformations of Hom-pre-Lie algebra structure. This new cohomology is a natural one by considering the structure map. We develop equivariant cohomology theory for a Hom-pre-Lie algebra equipped with a finite group action by formulating a proper notion of coefficients system for the equivariant cohomology. We also study the associated formal deformation theory for Hom-pre-Lie algebras in the equivariant context.
 \smallskip

 \smallskip

 \end{minipage}
 \end{center}
 \normalsize\vskip0.5cm

\section*{Introduction}
\def\theequation{\arabic{section}. \arabic{equation}}
\setcounter{equation} {0}
The notion of pre-Lie algebra (also known as left-symmetric algebra \cite{V63}) was introduced by Gerstenhaber \cite{G63} while studying the algebraic deformation theory of associative algebras. The pre-Lie algebras have a close relationship with Lie algebras. For example, a pre-Lie algebra gives rise to a Lie algebra via the commutator bracket, which is known as the sub-adjacent Lie algebra. Pre-Lie algebras have appeared in many areas of mathematics, for example, vertex algebras, quantum field theories, complex and symplectic structures on Lie groups and Lie algebras, etc., see \cite{MK03, CL01, CK00, LM88}, and references therein. A (left) pre-Lie algebra $\mathfrak{g}$ is a $\mathbb{K}$-vector space together with a binary operation $\c$ satisfying the following identity :
$$(a \c b) \c c - a \c (b \c c) = (b \c a) \c c - b \c (a \c c),~\text{for all}~a,b,c \in \mathfrak{g}.$$

The notion of Hom-pre-Lie algebras is a twisted analog of pre-Lie algebras, where the pre-Lie algebra identity is twisted by a self linear map, called the structure map. The study of Hom-version of this type of algebras is an interesting and active area of research these days and found some applications in mathematics as well as in physics. For the various type of Hom-algebras, see \cite{M10, M08, MS10, MZ18}.

Gertenhaber \cite{G63, G64} initiated the study of algebraic deformation theory for associative algebra. Later, following his ideas, deformation theory for various types of algebras has been studied, see \cite{B97, NR66, NR68, LST20, S20, saha19}. Hurle and Makhlouf studied cohomology and deformation theory for Hom-associative and Hom-Lie algebras in \cite{HM19, HM19glas}. In \cite{MS20}, Mukherjee and Saha studied cohomology and deformation theory for Hom-Leibniz algebras.

In this paper, we develop a new cohomology theory for Hom-pre-Lie algebras, we call this new cohomology $\alpha$-type cohomology of Hom-pre-Lie algebras. In \cite{LST20}, the authors defined cohomology and a one-parameter formal deformation for Hom-pre-Lie algebras without considering the structure map $\alpha$. In \cite{LST20}, they have only deformed the multiplication map, and ignored to deform the structure map. Thus, the study of cohomology and deformation in \cite{LST20} is incomplete. This is one of the motivation of writing this paper. In this paper, we deform both the multiplication operation and the structure map. We also develop the corresponding deformation cohomology of Hom-pre-Lie algebras. Recently, we have noticed that a similar cohomology and deformation theory were independently studied in \cite{LMS22}. 
In the next part, we develop  equivariant cohomology for Hom-pre-Lie algebras equipped with a finite group action and also introduce a one-parameter formal deformation theory for Hom-pre-Lie algebras in the equivariant contexts. To develop the equivariant cohomology, we closely follow Bredon's approach \cite{bredon67} for equivariant cohomology of topological $G$-spaces. Similar to the non-equivariant case we show that this equivariant cohomology is deformation cohomology of the equivariant deformation of Hom-pre-Lie algebras.

This paper is organized as follows: In Section \ref{sec 1}, we have discussed the definition, examples, representation, and a known cohomology of Hom-pre-Lie algebras which we shall use throughout the paper. In Section \ref{sec 2}, we develop a new cohomology for (multiplicative) Hom-pre-Lie algebras, we call this $\alpha$-type cohomology of Hom-pre-Lie algebras. In Section \ref{sec 3}, we introduce a one-parameter formal deformation theory of Hom-pre-Lie algebras considering the structure map $\alpha$. We show that infinitesimal of a deformation is a $2$-cocycle of the $\alpha$-type cohomology. We also discuss obstructions and rigidity of Hom-pre-Lie algebra deformations and show that how they are related to our new cohomology. In Section \ref{sec 4}, we develop the equivariant cohomology of Hom-pre-Lie algebras equipped with finite group action. In Section \ref{sec 5}, we introduce a one-parameter formal deformation theory for Hom-pre-Lie algebras in the equivariant contexts. We show some results for equivariant deformation of Hom-pre-Lie algebras analoguous to the non-equivariant case.

\section{Preliminaries} \label{sec 1}
\def\theequation{\arabic{section}.\arabic{equation}}
\setcounter{equation} {0}
In this section, we recall the definition, representation, and cohomology of Hom-pre-Lie algebras following \cite{LST20, SL17}. Throughout the paper, $\mathbb{K}$ denotes a field of characteristics zero.
\begin{definition}
A Hom-Lie algebra is a triple $(g, [~, ~], \alpha)$ consisting of a linear space $g$, a skew-
symmetric bilinear map $[~, ~] : \wedge^2g\rightarrow g$ and an algebra morphism $\alpha : g \rightarrow g$, satisfying:
\begin{eqnarray}
[\alpha(x), [y, z]] + [\alpha(y), [z, x]] + [\alpha(z), [x, y]] = 0,~\text{for all}~x, y, z \in g.
\end{eqnarray}
\end{definition}

\begin{definition}
A Hom-pre-Lie algebra $(A, \cdot, \alpha)$ is a vector space $A$ equipped with a bilinear product
$\cdot : A\times A \rightarrow A,~\text{and a linear map} ~\alpha : A \to A$, satisfying
\begin{enumerate}
\item ($a \cdot b) \cdot \alpha(c) - \alpha(a)\cdot (b \cdot c) = (b \cdot a) \cdot \alpha(c)- \alpha(b) \cdot (a \cdot c)$, for all $a, b, c\in A$.
\item $\alpha(x\cdot y) = \alpha(x) \cdot\alpha(y)$, for all $a,b \in A.$
\end{enumerate}
A Hom-pre-Lie algebra $(A, \c, \alpha)$ is called regular if $\alpha$ is invertible.
\end{definition}

%We will  often write $(A, \nu, \alpha)$ instead of $(A, \cdot, \alpha)$, where $\nu : A\otimes A\rightarrow A$ is given by $\nu(x\otimes y) = x\cdot y$.

\begin{definition}
Let $(A_1,\c_1, \alpha_1)$ and $(A_2,\c_2, \alpha_2)$ be two Hom-pre-Lie algebras. A morphism from $A_1$ to $A_2$ is a $\mathbb{K}$-linear map $\phi : A_1 \to A_2$ such that 
$$\phi(a\c_1b) =\phi(a)\c_2 \phi(b),~~\phi\circ \alpha_1 = \alpha_2 \circ \phi,~\text{for all}~a,b\in A_1.$$
\end{definition}

\begin{example}
Let $A$ be a two dimensional vector space with basis $\lbrace a_1, a_2 \rbrace$. Define a multiplication on $A$ by
$$a_1\c a_1=0,~a_1\c a_2 =0,~a_2\c a_1 =a_1,~a_2\c a_2 =a_1 +a_2.$$
Let $\alpha$ be a $\mathbb{K}$-linear map on $A$ defined as $\alpha(a_1)=a_1,~\alpha(a_2)= a_1+a_2$. Then $(A,\c, \alpha)$ is a regular Hom-pre-Lie algebra.
\end{example}

\begin{definition}
A representation of a Hom-Lie algebra $(g, [\c, \c], \alpha)$ on a
vector space $V$ with respect to $\beta \in gl(V )$ is a linear map $\rho : g \rightarrow gl(V)$, such that for all $x, y \in g$,
the following equalities are satisfied:
\begin{eqnarray}
&&\rho(\alpha(x))\circ \beta = \beta \circ \rho(x),\\
&&\rho([x, y]) \circ \beta = \rho(\alpha(x))\circ \rho(y) - \rho(\alpha(y))\circ \rho(x).
\end{eqnarray}
\end{definition}

Let $(A, \cdot, \alpha)$ be a Hom-pre-Lie algebra. We always assume that it is regular, i.e. $\alpha$ is invertible.
The commutator $[x, y]_C = x \cdot y - y \cdot x$ gives a Hom-Lie algebra $(A, [\cdot, \cdot]_C, \alpha)$, which is denoted by
$A^C$ and called the sub-adjacent Hom-Lie algebra of $(A, \cdot, \alpha)$.
\begin{definition}
A representation of a Hom-pre-Lie algebra $(A, \cdot, \alpha)$ on a vector space $V$ with
respect to  $\beta \in gl(V)$ consists of a pair $(\rho, \mu)$, where $\rho : A \rightarrow gl(V)$ is a representation of the
sub-adjacent Hom-Lie algebra $A^C$ on $V$ with respect to $\beta \in gl(V)$, and $\mu: A \rightarrow gl(V)$ is a linear
map, for all $a, b \in A$, satisfying:
\begin{eqnarray}
&&\beta \circ \mu(a) = \mu(\alpha(a)) \circ \beta,\\
&&\mu(\alpha(b)) \circ \mu(a)-\mu(a \cdot b) \circ \beta = \mu(\alpha(b))\circ \rho(a)-\rho(\alpha(a)) \circ \mu(b).
\end{eqnarray}
\end{definition}
If $V$ is a presentation of a Hom-pre-Lie algebra $A$ as defined above, then we denote it as $(V, \rho, \mu, \beta).$
\medskip

Let $(A, \c, \alpha)$ be a Hom-pre-Lie algebra. Then $(A, \alpha)$ is a representation of itself via the maps $L, R : A \to \mathfrak{gl}(A)$ defined as follows:
$$L_a(b)=a\c b,~~~R_a(b)=b\c a,~\text{for all}~a, b\in A.$$
It is easy to check that $(A, L, R, \alpha)$ is a representation of $(A,\c, \alpha)$. This representation of $A$ is called the regular representation. In this paper, we will use regular representation only.

Let $(V, \rho, \mu, \beta)$ be a representation of a Hom-pre-Lie algebra $(A, \cdot, \alpha)$. The  cohomology of $A$ with coefficients in $V$ is the
cohomology of the cochain complex $\{ C^*(A, V), \partial \}$, where $C^n (A, V) = \lbrace \phi \in \text{Hom}( A^{\otimes n}, V ) \mid \beta\circ \phi = \phi\circ \alpha^{\otimes n} \rbrace, (n \geq 0)$ and the
coboundary operator $\partial:C^n(A, V ) \rightarrow C^{n+1}(A, V )$ given by
\begin{eqnarray}\label{prelim coho}
&&(\partial f)(a_1,\ldots, a_{n+1})\\ \nonumber
&=&\sum^{n}_{i=1}(-1)^{i+1}\rho(\alpha^{n-1}(x_i))f(a_1,\ldots, \widehat{a_i}, \ldots, a_{n+1}) \\ \nonumber
&&+ \sum^{n}_{i=1}(-1)^{i+1}\mu(\alpha^{n-1}(a_{n+1}))f(a_1,\ldots,\widehat{a_i}, \ldots,  a_{n}, a_i)\\ \nonumber
&&-\sum^{n}_{i=1}(-1)^{i+1} f(\alpha(a_1),\ldots, \widehat{a_i}, \ldots,\alpha( a_n),  a_i\c a_{n+1})\\ \nonumber
&&+\sum_{1\leq i< j\leq n} (-1)^{i+j} f([a_i, a_j]_C, \alpha(a_1),\ldots, \widehat{a_i}, \ldots, \widehat{a_j},\ldots, \alpha(a_{n+1})),
\end{eqnarray}
for $a_1, \ldots, a_{n+1}\in A$. The corresponding cohomology groups are denoted by $H^*(A, V).$ We will see that this cohomology is not a deformation cohomology for Hom-pre-Lie algebras.

\section{\texorpdfstring{$\alpha$}{alpha}-type cohomology for Hom-pre-Lie algebras}\label{sec 2}
\def\theequation{\arabic{section}. \arabic{equation}}
\setcounter{equation} {0}
In this section,  we define an $\alpha$-type cohomology for Hom-pre-Lie algebras. Let $(V, \rho, \mu, \beta)$ be a representation of a Hom-pre-Lie algebra $(A, \nu, \alpha)$.  

Define $\nu: A \to A$ by $\nu(a, b)= a\c b$. Using this notation, we denote a Hom-pre-Lie algebra $(A, \c, \alpha)$ as $(A,\nu, \alpha)$.
Then the complex for the cohomology of $A$ with values in $V$ is given by
\begin{align*}
& \widetilde{C^n}(A,V) = \widetilde{C^n_{\nu}} (A, V) \oplus \widetilde{C_{\alpha}}^{n-1} (A, V) = \text{Hom}(A^{\otimes n}, V ) \oplus \text{Hom}( A^{\otimes n-1}, V ), ~\text{for}~n\geq 1,\\
&  \widetilde{C^n}(A,V) = \lbrace 0 \rbrace, ~\text{for}~n\leq 0.
\end{align*}
Note that for convience, we take $\widetilde{C^{1}_{\alpha}} (A, V) = \text{Hom}( A^0, V ) = \text{Hom}( \mathbb{K}, V )= \lbrace 0 \rbrace$.
We write $(\varphi, \psi )$ with $\varphi \in \widetilde{C^{\bullet}_{\nu}} (A, V)$ and  $\psi \in \widetilde{C^{\bullet}_{\alpha}} (A, V)$ for an element in $\widetilde{C}^\bullet(A,V)$.

We define four maps, with domain and range given in the following diagram:

\begin{equation*}
\xymatrix{
\widetilde{C_\nu}^n \ar"3,3"^-{~~~\widetilde{\partial}_{\alpha \nu}} \ar[rr]^-{\widetilde{\partial}_{\nu\nu}}& & \widetilde{C_\nu}^{n+1}  \\
\bigoplus & &\bigoplus \\
\widetilde{C_{\alpha}^n} \ar"1,3"^-{\widetilde{\partial}_{\nu \alpha}~} \ar[rr]^-{\widetilde{\partial}_{\alpha\alpha}} & & \widetilde{C_{\alpha}}^{n+1},}
\end{equation*}
\begin{eqnarray}
(\widetilde{\partial}_{\nu\nu} \varphi)(a_1,\ldots, a_{n+1})&=&\sum^{n}_{i=1}(-1)^{i+1}\rho(\alpha^{n-1}(a_i))\varphi(a_1,\ldots, \widehat{a_i}, \ldots, a_{n+1}) \\
&&+ \sum^{n}_{i=1}(-1)^{i+1}\mu(\alpha^{n-1}(a_{n+1}))\varphi(a_1,\ldots,\widehat{a_i},\ldots,  a_{n}, a_i)\nonumber\\
&&-\sum^{n}_{i=1}(-1)^{i+1} \varphi(\alpha(a_1),\ldots, \widehat{a_i}, \ldots,\alpha( a_n),  a_i\c a_{n+1})\nonumber\\
&&+\sum_{1\leq i< j\leq n} (-1)^{i+j} \varphi([a_i, a_j]_C, \alpha(a_1),\ldots, \widehat{a_i}, \ldots, \widehat{a_j},..., \alpha(a_{n+1})),\nonumber\\
(\widetilde{\partial}_{\alpha\alpha} \psi)(a_1,\ldots, a_{n})&=&\sum^{n}_{i=1}(-1)^{i+1}\rho(\alpha^{n-1}(a_i))\psi(a_1,\ldots, \widehat{a_i},\ldots, a_{n}) \\
&&+ \sum^{n-1}_{i=1}(-1)^{i+1}\mu(\alpha^{n-1}(a_{n}))\psi(a_1,\ldots,\widehat{a_i}, \ldots,  a_{n-1}, a_i)\nonumber\\
&&-\sum^{n-1}_{i=1}(-1)^{i+1} \psi(\alpha(a_1),\ldots, \widehat{a_i},\ldots,\alpha( a_{n-1}),  a_i\c a_{n})\nonumber\\
&&+\sum_{1\leq i< j\leq n-1} (-1)^{i+j} \psi([a_i, a_j]_C, \alpha(a_1),\ldots, \widehat{a_i}, \ldots, \widehat{a_j},\ldots, \alpha(a_{n})),\nonumber
\end{eqnarray}
\begin{eqnarray}
(\widetilde{\partial}_{\nu \alpha} \varphi)(a_1,\ldots, a_{n})&=& \beta(\varphi(a_1,\ldots, a_{n}))-\varphi(\alpha(a_1),\ldots, \alpha(a_{n}))\\
(\widetilde{\partial}_{\alpha \nu} \psi)(a_1,\ldots, a_{n+1})&=&\sum^{n}_{i=1}(-1)^{i+1} \mu\circ\alpha^{n-2}(a_i\c a_{n+1})\psi(a_1,\ldots, \widehat{a_i}, \ldots,a_n)\nonumber\\
&&-\sum_{1\leq i< j\leq n} (-1)^{i+j}\rho( [\alpha^{n-2}(a_i), \alpha^{n-2}(a_j)]_C)\psi(a_1,\ldots, \widehat{a_i},\ldots, \widehat{a_j},\ldots, a_{n+1}),\nonumber
\end{eqnarray}
where $a_1, \ldots, a_{n+1}\in A$.
The sign given by $(-1)^{\cdot}$ is always determined by the permutation of the $x_{i\cdot}$.
We have the following main theorem.
\begin{theorem}
Let $(V,\rho, \mu, \beta)$ be a representation of a Hom-pre-Lie algebra $(A, \nu, \alpha)$. Further, let
$\widetilde{\partial} :\widetilde{C}^n(A,V)\rightarrow \widetilde{C}^{n+1}(A,V)$ be a map defined by $\widetilde{\partial}(\varphi, \psi ) = (\widetilde{\partial}_{\nu\nu} \varphi - \widetilde{\partial}_{\alpha\nu} \psi, \widetilde{\partial}_{\nu\alpha} \varphi-\widetilde{\partial}_{\alpha\alpha} \psi)$. Then the pair $(\widetilde{C}^\bullet(A,V), \widetilde{\partial})$ is a cohomology complex.
\end{theorem}
\begin{proof}
In order to prove that $\widetilde{\partial}^2=0$, we  will give it here to some extent.
 Showing $\widetilde{\partial}^2=0$ is same as the showing following equations:
\begin{align*}
& \widetilde{\partial}_{\nu \nu} \widetilde{\partial}_{\nu \nu} + \widetilde{\partial}_{\nu \nu} \widetilde{\partial}_{\alpha \nu} - \widetilde{\partial}_{\alpha \nu} \widetilde{\partial}_{\alpha \alpha} - \widetilde{\partial}_{\alpha \nu} \widetilde{\partial}_{\nu \alpha} = 0,\\
& -\widetilde{\partial}_{\alpha \alpha} \widetilde{\partial}_{\alpha \alpha} + \widetilde{\partial}_{\nu \alpha} \widetilde{\partial}_{\alpha \nu} - \widetilde{\partial}_{\alpha \alpha} \widetilde{\partial}_{\nu \alpha} + \widetilde{\partial}_{\nu \alpha} \widetilde{\partial}_{\nu \nu} = 0.
\end{align*}
For this we verify the following equations
\begin{align*}
& \widetilde{\partial}_{\nu \nu} \widetilde{\partial}_{\nu \nu}= \widetilde{\partial}_{\alpha \nu} \widetilde{\partial}_{\nu \alpha},\\
& \widetilde{\partial}_{\nu \nu} \widetilde{\partial}_{\alpha \nu} = \widetilde{\partial}_{\alpha \nu} \widetilde{\partial}_{\alpha \alpha},\\
& \widetilde{\partial}_{\alpha \alpha} \widetilde{\partial}_{\alpha \alpha} = \widetilde{\partial}_{\nu \alpha} \widetilde{\partial}_{\alpha \nu},\\
& \widetilde{\partial}_{\nu \alpha} \widetilde{\partial}_{\nu \nu} = \widetilde{\partial}_{\alpha \alpha} \widetilde{\partial}_{\nu \alpha}.
\end{align*}
We shall only verify the first equality as computations of such verifications is very lengthy. The verification of the other three equalities can be done in a similar manner.

\begin{eqnarray}
&&\widetilde{\partial}_{\nu\nu} \widetilde{\partial}_{\nu\nu} \varphi(a_1,\ldots, a_{n+2})\nonumber\\
&=&\sum^{n+1}_{i=1}(-1)^{i+1}\rho(\alpha^{n}(a_i))\widetilde{\partial}_{\nu\nu}\varphi(a_1,\ldots, \widehat{a_i}, \ldots, a_{n+2}) \nonumber\\
&&+ \sum^{n+1}_{i=1}(-1)^{i+1}\mu(\alpha^{n}(a_{n+2}))\widetilde{\partial}_{\nu\nu}\varphi(a_1,\ldots,\widehat{a_i},\ldots,  a_{n+1}, a_i)\nonumber\\
&&-\sum^{n+1}_{i=1}(-1)^{i+1} \widetilde{\partial}_{\nu\nu}\varphi(\alpha(a_1),\ldots, \widehat{a_i}, \ldots,\alpha( a_{n+1}),  a_i\c a_{n+2})\nonumber\\
&&+\sum_{1\leq i< j\leq n+1} (-1)^{i+j} \widetilde{\partial}_{\nu\nu}\varphi([a_i, a_j]_C, \alpha(a_1),\ldots, \widehat{a_i}, \ldots, \widehat{a_j},..., \alpha(a_{n+2})),\nonumber\\
&=&\sum^{n+1}_{i=1}\sum^{n}_{\substack{j=1\\ j< i}} (-1)^{i+j} \rho(\alpha^{n}(a_i))\rho(\alpha^{n-1}(a_j))\varphi(a_1,\ldots, \widehat{a_j}, \ldots,\widehat{a_i} \ldots a_{n+2})\\
&&+\sum^{n+1}_{i=1}\sum^{n}_{\substack{j=1\\ j> i}} (-1)^{i+j-1} \rho(\alpha^{n}(a_i))\rho(\alpha^{n-1}(a_j))\varphi(a_1,\ldots, \widehat{a_i}, \ldots,\widehat{a_j} \ldots a_{n+2})\\
&& +\sum^{n+1}_{i=1}\sum^{n}_{\substack{j=1\\ j< i}} (-1)^{i+j}\rho(\alpha^{n}(a_i)) \mu(\alpha^{n-1}(a_{n+2})) \varphi(a_1,\ldots, \widehat{a_j}, \ldots,\widehat{a_i} \ldots a_{n+1}, a_{j})\\
&&+\sum^{n+1}_{i=1}\sum^{n}_{\substack{j=1\\ j> i}} (-1)^{i+j-1} \rho(\alpha^{n}(a_i))\rho(\alpha^{n-1}(a_j))\varphi(a_1,\ldots, \widehat{a_i}, \ldots,\widehat{a_j} \ldots a_{n+1}, a_j)\\
&& -\sum^{n+1}_{i=1}\sum^{n}_{\substack{j=1\\ j< i}} (-1)^{i+j}\rho(\alpha^{n}(a_i))\varphi(\alpha(a_1),\ldots, \widehat{a_j}, \ldots,\widehat{a_i} \ldots \alpha(a_{n+1}), a_{j}\cdot a_{n+2})\\
&& -\sum^{n+1}_{i=1}\sum^{n}_{\substack{j=1\\ j> i}} (-1)^{i+j-1}\rho(\alpha^{n}(a_i))\varphi(\alpha(a_1),\ldots, \widehat{a_i}, \ldots,\widehat{a_j} \ldots \alpha(a_{n+1}), a_{j}\cdot a_{n+2})\\
&& -\sum^{n+1}_{i=1}\sum_{j< k<i}(-1)^{i+j+k}\rho(\alpha^{n}(a_i))\varphi([a_j, a_k]_C, \alpha(a_1),\ldots, \widehat{a_j}, \ldots, \widehat{a_k},\ldots, \widehat{a_i}\ldots, \alpha(a_{n+2}))\nonumber\\
&&\\
&& +\sum^{n+1}_{i=1}\sum_{ j< i<k}(-1)^{i+j+k}\rho(\alpha^{n}(a_i))\varphi([a_j, a_k]_C, \alpha(a_1),\ldots, \widehat{a_j}, \ldots, \widehat{a_i},\ldots, \widehat{a_k}\ldots, \alpha(a_{n+2}))\nonumber\\
&&\\
&&-\sum^{n+1}_{i=1}\sum_{ i< j<k}(-1)^{i+j+k}\rho(\alpha^{n}(a_i))\varphi([a_j, a_k]_C, \alpha(a_1),\ldots, \widehat{a_i}, \ldots, \widehat{a_j},\ldots, \widehat{a_k}\ldots, \alpha(a_{n+2}))\nonumber\\
&&
\end{eqnarray}
\begin{eqnarray}
&& +\sum^{n+1}_{i=1}\sum^{n}_{\substack{j=1\\ j< i}} (-1)^{i+j} \mu(\alpha^{n}(a_{n+2}))\rho(\alpha^{n-1}(a_j))\varphi(a_1,\ldots, \widehat{a_j}, \ldots,\widehat{a_i} \ldots a_{n+1}, a_i)\\
&& +\sum^{n+1}_{i=1}\sum^{n}_{\substack{j=1\\ j> i}} (-1)^{i+j-1} \mu(\alpha^{n}(a_{n+2}))\rho(\alpha^{n-1}(a_j))\varphi(a_1,\ldots, \widehat{a_j}, \ldots,\widehat{a_i} \ldots a_{n+1}, a_i)~~~~\\
&&+\sum^{n+1}_{i=1}\sum^{n}_{\substack{j=1\\ j< i}} (-1)^{i+j} \mu(\alpha^{n}(a_{n+2}))\mu(\alpha^{n-1}(a_i))\varphi(a_1,\ldots, \widehat{a_j}, \ldots,\widehat{a_i} \ldots a_{n+1}, a_j)\\
&& +\sum^{n+1}_{i=1}\sum^{n}_{\substack{j=1\\ j> i}} (-1)^{i+j-1} \mu(\alpha^{n}(a_{n+2}))\mu(\alpha^{n-1}(a_i))\varphi(a_1,\ldots, \widehat{a_i}, \ldots,\widehat{a_j} \ldots a_{n+1}, a_j)~~~~\\
&& -\sum^{n+1}_{i=1}\sum^{n}_{\substack{j=1\\ j< i}} (-1)^{i+j} \mu(\alpha^{n}(a_{n+2}))\varphi(\alpha(a_1),\ldots, \widehat{a_j}, \ldots,\widehat{a_i} \ldots \alpha(a_{n+1}), a_j\cdot a_i)\\
&& -\sum^{n+1}_{i=1}\sum^{n}_{\substack{j=1\\ j>i}} (-1)^{i+j-1} \mu(\alpha^{n}(a_{n+2}))\varphi(\alpha(a_1),\ldots, \widehat{a_i}, \ldots,\widehat{a_j} \ldots \alpha(a_{n+1}), a_j\cdot a_i)\\
&&-\sum^{n+1}_{i=1}\sum_{j< k<i}(-1)^{i+j+k}\mu(\alpha^{n}(a_{n+2}))\varphi([a_j, a_k]_C, \alpha(a_1),\ldots, \widehat{a_j}, \ldots, \widehat{a_k},\ldots, \widehat{a_i}\ldots, \alpha(a_{n+1}),\alpha(a_{i}))\nonumber\\
&&\\
&&+\sum^{n+1}_{i=1}\sum_{j< i<k}(-1)^{i+j+k}\mu(\alpha^{n}(a_{n+2}))\varphi([a_j, a_k]_C, \alpha(a_1),\ldots, \widehat{a_j}, \ldots, \widehat{a_i},\ldots, \widehat{a_k}\ldots, \alpha(a_{n+1}),\alpha(a_{i}))\nonumber\\
&&\\
&&-\sum^{n+1}_{i=1}\sum_{i< j<k}(-1)^{i+j+k}\mu(\alpha^{n}(a_{n+2}))\varphi([a_j, a_k]_C, \alpha(a_1),\ldots, \widehat{a_i}, \ldots, \widehat{a_j},\ldots, \widehat{a_k}\ldots, \alpha(a_{n+1}),\alpha(a_{i}))\nonumber\\
&&\\
&& -\sum^{n+1}_{i=1}\sum^{n}_{\substack{j=1\\ j< i}} (-1)^{i+j} \rho(\alpha^{n}(a_{j}))\varphi(\alpha(a_1),\ldots, \widehat{a_j}, \ldots,\widehat{a_i} \ldots \alpha(a_{n+1}), a_i\cdot a_{n+2})\\
&& -\sum^{n+1}_{i=1}\sum^{n}_{\substack{j=1\\ j> i}} (-1)^{i+j-1} \rho(\alpha^{n}(a_{j}))\varphi(\alpha(a_1),\ldots, \widehat{a_i}, \ldots,\widehat{a_j} \ldots \alpha(a_{n+1}), a_i\cdot a_{n+2})\\
&& - \sum^{n+1}_{i=1}\sum^{n}_{\substack{j=1\\ j< i}} (-1)^{i+j} \mu(\alpha^{n-1}(a_{i})\cdot \alpha^{n-1}(a_{n+2}))\varphi(\alpha(a_1),\ldots, \widehat{a_j}, \ldots,\widehat{a_i} \ldots \alpha(a_{n+1}), \a(a_j))\\
&& - \sum^{n+1}_{i=1}\sum^{n}_{\substack{j=1\\ j> i}} (-1)^{i+j} \mu(\alpha^{n-1}(a_{i})\cdot \alpha^{n-1}(a_{n+2}))\varphi(\alpha(a_1),\ldots, \widehat{a_i}, \ldots,\widehat{a_j} \ldots \alpha(a_{n+1}), \a(a_j))
\end{eqnarray}
\begin{eqnarray}
&& +\sum^{n+1}_{i=1}\sum^{n}_{\substack{j=1\\ j< i}} (-1)^{i+j}\varphi(\alpha^2(a_1),\ldots, \widehat{a_j}, \ldots,\widehat{a_i} \ldots \alpha^2(a_{n+1}),\alpha(x_j)\cdot( a_i\cdot a_{n+2}))\\
&& +\sum^{n+1}_{i=1}\sum^{n}_{\substack{j=1\\ j> i}} (-1)^{i+j}\varphi(\alpha^2(a_1),\ldots, \widehat{a_i}, \ldots,\widehat{a_j} \ldots \alpha^2(a_{n+1}),\alpha(x_j)\cdot( a_i\cdot a_{n+2}))\\
&& +\sum^{n+1}_{i=1}\sum_{ j< k<i}(-1)^{i+j+k}\varphi([\alpha(a_j), \alpha(a_k)]_C, \alpha^2(a_1),\ldots, \widehat{a_j}, \ldots, \nonumber\\
&&\widehat{a_k},\ldots, \widehat{a_i}\ldots, \alpha^2(a_{n+1}),\alpha(a_{i})\cdot \alpha(a_{n+2}))\\
&&-\sum^{n+1}_{i=1}\sum_{ j< i<k}(-1)^{i+j+k}\varphi([\alpha(a_j), \alpha(a_k)]_C, \alpha^2(a_1),\ldots, \widehat{a_j}, \ldots, \nonumber\\
&&\widehat{a_i},\ldots, \widehat{a_k}\ldots, \alpha^2(a_{n+1}),\alpha(a_{i})\cdot \alpha(a_{n+2}))\\
&&+\sum^{n+1}_{i=1}\sum_{ i< j<k}(-1)^{i+j+k}\varphi([\alpha(a_j), \alpha(a_k)]_C, \alpha^2(a_1),\ldots, \widehat{a_i}, \ldots, \nonumber\\
&&\widehat{a_j},\ldots, \widehat{a_k}\ldots, \alpha^2(a_{n+1}),\alpha(a_{i})\cdot \alpha(a_{n+2}))\\
&& + \sum_{1\leq i< j\leq n+1}(-1)^{i+j} \rho([\alpha^{n-1}(a_i), \alpha^{n-1}(a_j)]_C)\varphi(\alpha(a_1),\ldots, \widehat{a_i}, \ldots,\widehat{a_j} \ldots \alpha(a_{n+1}), \a(a_{n+2}))~~~~~~~~~\\
&&  +\sum^{n+1}_{1\leq i\leq j}\sum_{k< i} (-1)^{i+j+k+1} \rho(\alpha^{n}(a_{k}))\varphi([a_i, a_j]_C, \alpha(a_1),\ldots, \widehat{a_k}\ldots, \widehat{a_i}, \ldots,\widehat{a_j} \ldots, \a(a_{n+2}))\\
&& +\sum^{n+1}_{1\leq i\leq j}\sum_{i< k< j} (-1)^{i+j+k} \rho(\alpha^{n}(a_{k}))\varphi([a_i, a_j]_C, \alpha(a_1),\ldots, \widehat{a_i}\ldots, \widehat{a_k}, \ldots,\widehat{a_j} \ldots, \a(a_{n+2}))\\
&& +\sum^{n+1}_{1\leq i\leq j}\sum_{j<k} (-1)^{i+j+k+1} \rho(\alpha^{n}(a_{k}))\varphi([a_i, a_j]_C, \alpha(a_1),\ldots, \widehat{a_i}\ldots, \widehat{a_j}, \ldots,\widehat{a_k} \ldots, \a(a_{n+2}))\\
&& +\sum^{n+1}_{1\leq i\leq j}(-1)^{i+j} \mu(\alpha^{n}(a_{n+2}))\varphi( \alpha(a_1),\ldots, \widehat{a_i}\ldots, \widehat{a_j},  \ldots, \a(a_{n+1}), [a_i, a_j]_C)\\
&& -\sum^{n+1}_{1\leq i\leq j}\sum_{k< i} (-1)^{i+j+k} \mu(\alpha^{n}(a_{n+2}))\varphi([a_i, a_j]_C, \alpha(a_1),\ldots, \widehat{a_k}\ldots, \widehat{a_i}, \ldots,\widehat{a_j} \ldots, \a(a_{k}))\\
&&+\sum^{n+1}_{1\leq i\leq j}\sum_{i< k< j} (-1)^{i+j+k} \mu(\alpha^{n}(a_{n+2}))\varphi([a_i, a_j]_C, \alpha(a_1),\ldots, \widehat{a_i}\ldots, \widehat{a_k}, \ldots,\widehat{a_j} \ldots, \a(a_{k}))\\
&& -\sum^{n+1}_{1\leq i\leq j}\sum_{j<k} (-1)^{i+j+k+1} \mu(\alpha^{n}(a_{n+2}))\varphi([a_i, a_j]_C, \alpha(a_1),\ldots, \widehat{a_i}\ldots, \widehat{a_j}, \ldots,\widehat{a_k} \ldots, \a(a_{k}))\\
&&-\sum^{n+1}_{1\leq i\leq j}(-1)^{i+j} \varphi(\alpha^2(a_1),\ldots, \widehat{a_i}\ldots, \widehat{a_j},\ldots, \alpha^2(a_{n+1}) [a_i, a_j]_C\cdot \a(a_{n+2}))
\end{eqnarray}
\begin{eqnarray}
&& +\sum^{n+1}_{1\leq i\leq j}\sum_{k<i}(-1)^{i+j+k}\varphi([\alpha(a_i), \alpha(a_j)]_C, \alpha^2(a_1),\ldots, \widehat{a_k}\ldots, \widehat{a_i}, \ldots,\nonumber\\
&&\widehat{a_j} \ldots, \a^2(a_{n+1}), \a(a_k)\cdot \a(a_{n+2}))\\
&&-\sum^{n+1}_{1\leq i\leq j}\sum_{i<k<j}(-1)^{i+j+k}\varphi([\alpha(a_i), \alpha(a_j)]_C, \alpha^2(a_1),\ldots, \widehat{a_i}\ldots, \widehat{a_k}, \ldots,\nonumber\\
&&\widehat{a_j} \ldots, \a^2(a_{n+1}), \a(a_k)\cdot \a(a_{n+2}))\\
&&+\sum^{n+1}_{1\leq i\leq j}\sum_{j<k}(-1)^{i+j+k}\varphi([\alpha(a_i), \alpha(a_j)]_C, \alpha^2(a_1),\ldots, \widehat{a_i}\ldots, \widehat{a_j}, \ldots,\nonumber\\
&&\widehat{a_k} \ldots, \a^2(a_{n+1}), \a(a_k)\cdot \a(a_{n+2}))\\
&&-\sum^{n+1}_{1\leq i\leq j}\sum_{l< i} (-1)^{i+j+l} \varphi([[a_i, a_j]_C,\a(a_l)]_C,  \alpha^2(a_1),\ldots, \widehat{a_l}\ldots, \widehat{a_i}, \ldots,\widehat{a_j} \ldots, \a^2(a_{n+2}))~~~~~~~~~~~\\
&&+ \sum^{n+1}_{1\leq i\leq j}\sum_{ i<l<j} (-1)^{i+j+l} \varphi([[a_i, a_j]_C,\a(a_l)]_C,  \alpha^2(a_1),\ldots, \widehat{a_i}\ldots, \widehat{a_l}, \ldots,\widehat{a_j} \ldots, \a^2(a_{n+2}))~~~\\
&& -\sum^{n+1}_{1\leq i\leq j}\sum_{ j<l} (-1)^{i+j+l} \varphi([[a_i, a_j]_C,\a(a_l)]_C,  \alpha^2(a_1),\ldots, \widehat{a_i}\ldots, \widehat{a_j}, \ldots,\widehat{a_l} \ldots, \a^2(a_{n+2}))~~~\\
&& +\sum^{n+1}_{1\leq i\leq j}\sum_{ k<l<i}(-1)^{i+j+l+k} \varphi([\alpha(a_k), \alpha(a_l)]_C, [\alpha(a_i), \alpha(a_j)]_C,  \alpha^2(a_1),\ldots,\nonumber\\
&& \widehat{a_k}\ldots, \widehat{a_l}, \ldots,\widehat{a_i} \ldots ,\widehat{a_j} \ldots,  \a^2(a_{n+2}))\\
&& -\sum^{n+1}_{1\leq i\leq j}\sum_{ k<i<l<j}(-1)^{i+j+l+k} \varphi([\alpha(a_k), \alpha(a_l)]_C, [\alpha(a_i), \alpha(a_j)]_C,  \alpha^2(a_1),\ldots,\nonumber\\
&& \widehat{a_k}\ldots, \widehat{a_i}, \ldots,\widehat{a_l} \ldots ,\widehat{a_j} \ldots,  \a^2(a_{n+2}))\\
&& +\sum^{n+1}_{1\leq i\leq j}\sum_{ k<i<j<l}(-1)^{i+j+l+k} \varphi([\alpha(a_k), \alpha(a_l)]_C, [\alpha(a_i), \alpha(a_j)]_C,  \alpha^2(a_1),\ldots,\nonumber\\
&& \widehat{a_k}\ldots, \widehat{a_i}, \ldots,\widehat{a_j} \ldots ,\widehat{a_l} \ldots,  \a^2(a_{n+2}))\\
&& -\sum^{n+1}_{1\leq i\leq j}\sum_{ i<k<l<j}(-1)^{i+j+l+k} \varphi([\alpha(a_k), \alpha(a_l)]_C, [\alpha(a_i), \alpha(a_j)]_C,  \alpha^2(a_1),\ldots,\nonumber\\
&& \widehat{a_i}\ldots, \widehat{a_k}, \ldots,\widehat{a_l} \ldots ,\widehat{a_j} \ldots,  \a^2(a_{n+2}))\\
&& +\sum^{n+1}_{1\leq i\leq j}\sum_{ i<k<l<j}(-1)^{i+j+l+k} \varphi([\alpha(a_k), \alpha(a_l)]_C, [\alpha(a_i), \alpha(a_j)]_C,  \alpha^2(a_1),\ldots,\nonumber\\
&& \widehat{a_i}\ldots, \widehat{a_k}, \ldots,\widehat{a_j} \ldots ,\widehat{a_l} \ldots,  \a^2(a_{n+2}))\\
&& -\sum^{n+1}_{1\leq i\leq j}\sum_{ j<k<l}(-1)^{i+j+l+k} \varphi([\alpha(a_k), \alpha(a_l)]_C, [\alpha(a_i), \alpha(a_j)]_C,  \alpha^2(a_1),\ldots,\nonumber\\
&& \widehat{a_i}\ldots, \widehat{a_j}, \ldots,\widehat{a_k} \ldots ,\widehat{a_l} \ldots,  \a^2(a_{n+2})).
\end{eqnarray}
It is clearly to see that
\begin{eqnarray*}
&& (2.13)+(2.33) = (2.12)+(2.34)  \\
&=&(2.11)+(2.35)= (2.9)-(2.24)  \\
&=& (2.10)-(2.23)= (2.18)+(2.19)+(2.36) = 0,
\end{eqnarray*}
and
\begin{eqnarray*}
&&(2.20)+(2.39) = (2.21)+(2.38) \\
&=& (2.22)+(2.37)=(2.29)-(2.43) \\
&=& (2.30)-(2.42)= (2.31)-(2.41) = 0.
\end{eqnarray*}
Using the definition of Hom-pre-Lie algebra, we have
\begin{eqnarray*}
(2.27)+(2.28)-(2.40) = 0.
\end{eqnarray*}
Using the Hom-Jacobi identity,   we have (2.44)+(2.45)+(2.46) = 0. Finally, it easy to see that (2.47)+$\cdots$+(2.52) = 0.

Furthermore, we have 
\begin{eqnarray}
&&\widetilde{\partial}_{\alpha\nu} \widetilde{\partial}_{\nu\alpha}\varphi(a_1, \cdots, a_{n+2})\nonumber\\
&=& \sum^{n+1}_{i=1}(-1)^{i+1} \mu(\alpha^{n-1}(a_i\c a_{n+2}))\widetilde{\partial}_{\nu\alpha}\psi(a_1,\ldots, \widehat{a_i}, \ldots,a_{n+1})\nonumber\\
&&-\sum_{1\leq i< j\leq n+1} (-1)^{i+j}\rho( [\alpha^{n-1}(a_i), \alpha^{n-1}(a_j)]_C)\widetilde{\partial}_{\nu\alpha}\psi(a_1,\ldots, \widehat{a_i},\ldots, \widehat{a_j},\ldots, a_{n+2})\nonumber\\
&=& \sum^{n+1}_{i=1}(-1)^{i+1} \mu(\alpha^{n-1}(a_i\c a_{n+2}))\b\psi(a_1,\ldots, \widehat{a_i}, \ldots,a_{n+1})\\
&&-\sum^{n+1}_{i=1}(-1)^{i+1} \mu(\alpha^{n-1}(a_i\c a_{n+2}))\psi(\a(a_1),\ldots, \widehat{a_i}, \ldots, \a(a_{n+1}))\\
&&-\sum_{1\leq i< j\leq n+1} (-1)^{i+j}\rho( [\alpha^{n-1}(a_i), \alpha^{n-1}(a_j)]_C)\beta\psi(a_1,\ldots, \widehat{a_i},\ldots, \widehat{a_j},\ldots, a_{n+2})\\
&& +\sum_{1\leq i< j\leq n+1} (-1)^{i+j}\rho( [\alpha^{n-1}(a_i), \alpha^{n-1}(a_j)]_C)\psi(\a(a_1),\ldots, \widehat{a_i},\ldots, \widehat{a_j},\ldots, \a(a_{n+2})).~~~~~~~~~
\end{eqnarray}
Thus, we have  (2.5)+(2.6)+(2.32)-(2.55)-(2.56) = 0 and 
(2.7)+(2.14)+(2.16)-(2.25) -(2.53)-(2.54) = (2.8)+(2.15)+(2.17)-(2.26)-(2.53)-(2.54)  = 0.

Thus, $\widetilde{\partial_{\nu \nu}} \widetilde{\partial_{\nu \nu}} = \widetilde{\partial}_{\alpha \nu} \widetilde{\partial}_{\nu \alpha}$.
\end{proof}
In this paper, we will only consider cohomology of $A$ with coefficients over itself and denote it as $\widetilde{H^\ast} (A, A )$.

\begin{remark}
Observe that the $\alpha$-type cohomology for multiplicative Hom-pre-Lie algebras generalizes the cohomology defined in the Section \ref{sec 1} (cf. \ref{prelim coho}). For this consider only those elements in $\widetilde{C^n}(A, A)$ where second summand is zero, that is, $$\widetilde{C}_\alpha^n(A,A) = \lbrace 0 \rbrace.$$ Therefore, we have elements of the form $(\phi, 0)$. We define a subcomplex of $\widetilde{C^n}(A, A)$ as follows:
\begin{align*}
{C}_\alpha^n(A, A) & = \lbrace (\phi, 0) \in \widetilde{C^n}(A, A) \mid  \widetilde{\partial}_{\nu \alpha} \phi = 0 \rbrace\\
                                & = \lbrace \phi \in \widetilde{C}_\nu^n(A, A) \mid \beta \circ \phi = \phi \circ \alpha^{\otimes n}\rbrace.
\end{align*}
The map $\widetilde{\partial}_{\nu \nu}$ defines a diffential on this complex and this complex is same as the complex defined in the Section \ref{sec 1}. Thus, $\alpha$-type cohomology generalizes the cohomology developed in Section \ref{sec 1}.
\end{remark}
\section{Formal deformation theory of Hom-pre-Lie algebras}\label{sec 3}
In this section, we show that the $\alpha$-type cohomology for Hom-pre-Lie algebras controls one-parameter formal deformation of Hom-pre-Lie algebras. Unlike \cite{LST20}, our formal deformation takes into account the structure map $\alpha$.
\begin{definition} \label{formal deform defn}
\emph{
Let $(A, \nu, \alpha)$ be a Hom-pre-Lie algebra over $\mathbb{K}$. A one-parameter formal deformation of $A$ is given by a $\mathbb{K}[[t]]$-bilinear map $\nu_t : A[[t]] \times A[[t]] \to A[[t]]$, and a $\mathbb{K}[[t]]$-linear map $\alpha_t : A[[t]] \to A[[t]]$ of the forms
$$\nu_t = \sum_{i\geq 0} \nu_i t^i,~~~\alpha_t = \sum_{i\geq 0} \alpha_i t^i$$
such that
\begin{enumerate}
\item For all $i\geq 0$,~ $\nu_i : A \times A \to A$ is a $\mathbb{K}$-bilinear map, and $\alpha_i : A \to A$ is a $\mathbb{K}$-linear map.
\item $\nu_0 (a, b) = \nu(a,b)=a\c b$, is the the multiplication of $A$ and $\alpha_0 =\alpha$ is the given structure map on $A$.
\item  \label{deform cond 3}
$\nu_t(\nu_t(a,b), \alpha_t(c)) - \nu_t (\alpha_t(a), \nu_t(b,c)) = \nu_t(\nu_t(b,a), \alpha_t(c)) - \nu_t (\alpha_t(b), \nu_t(a,c))$ for all $a,b,c\in A.$
\item The map $\alpha_t$ is multiplicative, that is, $\nu_t (\alpha_t(a), \alpha_t(b)) = \alpha_t (\nu_t(a,b))$, for all $a, b\in A$. \label{deform cond 4}
\end{enumerate}
}
\end{definition}
From the Condition \ref{deform cond 3}, we have the following equations for all $n\geq 0$.
 \begin{align}\label{deform hom-pre-lie 2}
 \sum_{\substack{i+j+k = n\\i,j,k\geq 0}} \nu_i(\nu_j(a,b), \alpha_k(c)) - \nu_i (\alpha_j(a), \nu_k(b,c)) - \nu_i(\nu_j(b,a), \alpha_k(c)) + \nu_i (\alpha_j(b), \nu_k(a,c)) = 0.
 \end{align}
 The Condition \ref{deform cond 4} is equivalent to the following equation:
 \begin{align} \label{deform alpha}
 \sum_{\substack{i+j+k = n\\i,j,k\geq 0}} \nu_i (\alpha_j(a), \alpha_k(b)) - \sum_{\substack{i+j=n\\i,j\geq 0}} \alpha_i (\nu_j(a,b))=0.
 \end{align}
 For a Hom-pre-Lie algebra $(A, \nu, \alpha)$, an $\alpha_k$-associator is defined as
 \begin{align}
 \nu_i \circ_{\alpha_k} \nu_j (a, b, c) =  \nu_i(\nu_j(a,b), \alpha_k(c)) - \nu_i (\alpha_k(a), \nu_j(b,c)).
 \end{align}
 Thus, Equation \ref{deform hom-pre-lie 2} may be written in terms of associator $\alpha_j$ as follows:
 \begin{align}\label{deorm associator}
 \sum_{\substack{i+j+k=n\\ i,j,k\geq 0}} ( \nu_i \circ_{\alpha_j} \nu_k) (a, b, c) -   \sum_{\substack{i+j+k=n\\ i,j,k\geq 0}} (\nu_i \circ_{\alpha_k} \nu_j) (b, a, c) =0.
 \end{align}
 We can rewrite Equation \ref{deorm associator} as follows:
 \begin{align}\label{deform diff equ}
 (\widetilde{\partial}_{\nu\nu} \nu_n - \widetilde{\partial}_{\alpha\nu}\alpha_n)(a, b, c) =  \sum_{\substack{i+j+k=n\\0\leq  i,j,k\leq n-1}} ( \nu_i \circ_{\alpha_k} \nu_j) (a, b, c) -   \sum_{\substack{i+j+k=n\\ 0\leq i,j,k\leq n-1}} (\nu_i \circ_{\alpha_k} \nu_j) (b, a, c). 
 \end{align}
 Observe that
 \begin{align*}
 (\widetilde{\partial}_{\nu\nu} \nu_n)(a, b, c) = &\alpha(a)\c \nu_n(b,c) - \alpha(b)\c \nu_n(a,c) + \nu_n(b,a)\c \alpha(c) - \nu_n(a, b)\c \alpha(c) - \nu_n(\alpha(b), a\c c)  \\
& + \nu_n(\alpha(a), b\c c) - \nu_n(a\c b, \alpha(c)) + \nu_n(b\c a, \alpha(c)),
 \end{align*}
 and
 \begin{align*}
 (\widetilde{\partial}_{\alpha\nu}\alpha_n)(a, b, c) = \alpha_n(b) \c (a \c c) - \alpha_n(a) (b\c c)+ (a \c b) \c \alpha_n(c) - (b\c a)\c \alpha_n(c)
 \end{align*}
 The Equation \ref{deform alpha} is same as the following equation involving differentials.
 \begin{align}\label{deform alpha diff equ}
 (\widetilde{\partial}_{\nu\alpha}\nu_n- \widetilde{\partial}_{\alpha\alpha} \alpha_n)(a, b) = \sum_{\substack{i+j+k = n\\0\leq i,j,k\leq n-1}} \nu_i (\alpha_j(a), \alpha_k(b)) - \sum_{\substack{i+j=n\\ 0\leq i,j\leq n-1}} \alpha_i (\nu_j(a,b)).
 \end{align}
 Note that
 \begin{align*}
 (\widetilde{\partial}_{\alpha\alpha} \alpha_n) (a,b) = \alpha(a)\c \alpha_n(b) + \alpha_n(a)\c \alpha(b) - \alpha_n(a\c b).
 \end{align*}
 and 
  \begin{align*}
 (\widetilde{\partial}_{\nu\alpha} \nu_n) (a,b) = \alpha(\nu_n(a,b)) -\nu_n(\alpha(a), \alpha(b)).
 \end{align*}

 For $n=0$, from the Equation \ref{deform hom-pre-lie 2}, we have
 \begin{align*}
 (a.b).\alpha(c) - \alpha(a).(b.c) - (b.a).\alpha(c) + \alpha(b).(a.c) = 0.
 \end{align*}
This equation is same as Hom-pre-Lie identity.

 Again, for $n=0$, from the Equation \ref{deform alpha}, we have
 \begin{align*}
 \alpha(a).\alpha(b) - \alpha(a.b)=0.
 \end{align*}
 This shows $\alpha$ is multiplicative. Thus, for $n=0$ case, deformed Hom-pre-Lie algebra is same as the original Hom-pre-Lie algebra $A$.
 
 For $n=1$, from Equation \ref{deform hom-pre-lie 2}, we have
 \begin{align*}
& \alpha(a)\c \nu_1(b,c) - \alpha(b)\c \nu_1(a,c) + \nu_1(b,a)\c \alpha(c) - \nu_1(a, b)\c \alpha(c) - \nu_1(\alpha(b), a.c)- \alpha_1(b) \c (a \c c) \\
 &+ \alpha_1(a) (b\c c)- (a \c b) \c \alpha_1(c) + (b\c a)\c \alpha_n(c) =0.
 \end{align*}
 This is same as $(\tilde{\partial_{\nu\nu}} \nu_1 - \tilde{\partial_{\alpha\nu}}\alpha_1)(a, b, c)=0$.
 
  For $n=1$, from Equation \ref{deform alpha}, we also have
  $$ \alpha(a)\c \alpha_1(b) + \alpha_1(a)\c \alpha(b) - \alpha_1(a\c b) - \alpha(\nu_1(a,b)) + \nu_1(\alpha(a), \alpha(b)).$$
  Again, this is same as $(\widetilde{\partial}_{\nu\alpha}\nu_1 - \widetilde{\partial}_{\alpha\alpha} \alpha_1)(a, b) = 0.$
 \begin{definition}
 The pair $(\nu_1,\alpha_1)$ is called the infinitesimal of the given deformation $(\nu_t, \alpha_t)$. More generally, assume that $(\nu_n, \alpha_n)$ is the first non-zero term after $(\nu_0,\alpha_0)$ of $(\nu_t, \alpha_t)$. Then such  pair $(\nu_n, \alpha_n)$ is called the $n$th infinitesimal of $(\nu_t, \alpha_t)$.
 \end{definition}
 Therefore, from the above discussion we have the following theorem.
 \begin{theorem}\label{infinitesimal thm}
 Let $(\nu_t,\alpha_t)$ be a one-parameter formal deformation of the Hom-pre-Lie algebra $A$. Then the inifinitesimal of a given deformation of $A$ is a $2$-cocycle of the $\alpha$-type cohomology of $A$.
 \end{theorem}
 
 \subsection{Obstructions of deformation}
 In this section, we show how $\alpha$-type cohomology control obstructions of given deformations.
 \begin{definition}
 \emph{
 A deformation of order $n$ of $A$ is given by a $\mathbb{K}[[t]]$-bilinear map $\nu_t : A[[t]]\times A[[t]] \to A[[t]]$ and a $\mathbb{K}[[t]]$-linear map $\alpha_t : A[[t]] \to A[[t]]$ of the forms
 $$\nu_t = \sum^{n}_{i=0} \nu_i t^i,~~~\alpha_t = \sum^n_{i=0} \alpha_i t^i,$$
 such that the pair $(\nu_t, \alpha_t)$ satisfies the definition of the formal deformation of $A$ (mod $t^{n+1}$).
 }
 \end{definition}
 We say a deformation $(\nu_t, \alpha_t)$ of order $n$ is extendable to a deformation of order $(n+1)$ if there exists $\nu_{n+1} \in \widetilde{C_\nu}^{n+1}(A,A)$ and $\alpha_{n+1} \in \widetilde{C_\alpha}^{n+1}(A,A)$ such that 
 $$\bar{\nu_t} = \nu_t +\nu_{n+1}t^{n+1},~~~\bar{\alpha_t} = \alpha_t +\alpha_{n+1}t^{n+1},$$
 and the pair $(\bar{\nu_t},\bar{\alpha_t})$ satisfies the Definition \ref{formal deform defn} mod($t^{n+2}$).
  From a deformation $(\nu_t, \alpha_t)$ of order $(n+1)$, we have the following equations:
   \begin{align*}
& \sum_{\substack{i+j+k = n+1\\i,j,k\geq 0}} \nu_i(\nu_j(a,b), \alpha_k(c)) - \nu_i (\alpha_j(a), \nu_k(b,c)) - \nu_i(\nu_j(b,a), \alpha_k(c)) + \nu_i (\alpha_j(b), \nu_k(a,c)) = 0,\\
& \sum_{\substack{i+j+k = n+1\\i,j,k\geq 0}} \nu_i (\alpha_j(a), \alpha_k(b)) - \sum_{\substack{i+j=n+1\\i,j\geq 0}} \alpha_i (\nu_j(a,b))=0.
 \end{align*}
 We can re-write the above equations as follows:
 \begin{align*}
& (\widetilde{\partial}_{\nu\nu} \nu_n - \widetilde{\partial}_{\alpha\nu}\alpha_n)(a, b, c) =  \sum_{\substack{i+j+k=n+1\\0\leq  i,j,k\leq n}} ( \nu_i \circ_{\alpha_j} \nu_k) (a, b, c) -   \sum_{\substack{i+j+k=n+1\\ 0\leq i,j,k\leq n}} (\nu_i \circ_{\alpha_j} \nu_k) (b, a, c),\\ 
& (\widetilde{\partial}_{\nu\alpha}\nu_n - \widetilde{\partial}_{\alpha\alpha} \alpha_n)(a, b) = \sum_{\substack{i+j+k = n+1\\0\leq i,j,k\leq n}} \nu_i (\alpha_j(a), \alpha_k(b)) - \sum_{\substack{i+j=n+1\\ 0\leq i,j\leq n}} \alpha_i (\nu_j(a,b)).
 \end{align*}
 We define the $n$th obstruction to extend a deformation of order $n$ to $n+1$ as $O^n(A):= (O^n_\nu (A), O^n_\alpha (A))$, where
 \begin{align}
& O^n_\nu (A)(a,b,c)= \sum_{\substack{i+j+k=n+1\\0\leq  i,j,k\leq n}} ( \nu_i \circ_{\alpha_j} \nu_k) (a, b, c) -   \sum_{\substack{i+j+k=n+1\\ 0\leq i,j,k\leq n}} (\nu_i \circ_{\alpha_j} \nu_k) (b, a, c),\\
& O^n_\alpha (A)(a,b,c)= \sum_{\substack{i+j+k = n+1\\0\leq i,j,k\leq n}} \nu_i (\alpha_j(a), \alpha_k(b)) - \sum_{\substack{i+j=n+1\\ 0\leq i,j\leq n}} \alpha_i (\nu_j(a,b)).
 \end{align}
 Note that $O^n(A) \in \widetilde{C}^3(A,A)$.
 \begin{theorem}
 A deformation $(\nu_t, \alpha_t)$ of order $n$ is extendable to a deformation of order $n+1$ if and only if the cohomology class of $O^n(A)$ vanishes.
 \end{theorem} 
 \begin{proof}
 Suppose a deformation $(\nu_t, \alpha_t)$ of order $n$ extends a deformation of order $n+1$. From the definition of obstruction, we have
 $$O^n(A):= (O^n_\nu (A), O^n_\alpha (A)) = \partial (\nu_{n+1}, \alpha_{n+1}).$$
 Thus, $\partial (O^n(A))= 0$, that is, the cohomology class of the obstruction vanishes.
 
Conversely, suppose the cohomology class of $O^n(A)$ vanishes, that is,
$$O^n(A)= \partial(\nu_{n+1}, \alpha_{n+1}),$$ 
for some $2$-cochain $(\nu_{n+1}, \alpha_{n+1})$. We define an $(n+1)$-deformation extending an $n$-deformation $(\nu_t, \alpha_t)$ as follows:
\begin{align*}
&\bar{\nu_t} = \nu_t + \nu_{n+1}t^{n+1},\\
&\bar{\alpha_t} = \alpha_t +\alpha_{n+1}t^{n+1}.
\end{align*}
It is easy to check that $(\bar{\nu_t}, \bar{\alpha_t})$ defines an $(n+1)$-deformation.
 \end{proof}
 
 \begin{cor}
 The obstruction of extending an infinitesimal deformation lies in $\widetilde{H^3}(A, A)$.
 \end{cor}
 
 \subsection{Equivalent deformation and cohomology}
 Let $(\nu_t,\alpha_t)$ and $(\nu^{'}_t,\alpha^{'}_t)$ be two formal deformations of $A$, where 
 $\nu^{'}_t = \sum_{i\geq 0} \nu^{'}_i t^i,~~~\alpha^{'}_t = \sum_{i\geq 0} \alpha^{'}_i t^i$.
 \begin{definition}
 Two deformations $(\nu_t,\alpha_t)$ and $(\nu^{'}_t,\alpha^{'}_t)$  are said to be equivalent if there exists $\mathbb{K}[[t]]$-isomorphism $\Psi_t : A[[t]] \to A[[t]]$ of the form $\Psi_t = \sum_{i\geq 0} \psi_i t^i$ such that $\psi_0= \text{id}$, and $\psi_i : A \to A$ are $\mathbb{K}$-linear maps for all $i\geq 0$, satisfying the following conditions:
 \begin{align}\label{cond equiv defor}
 \Psi_t \circ \nu^{'}_t = \nu_t \circ (\Psi_t \otimes \Psi_t),~\text{and}~ \alpha_t\circ \Psi_t = \Psi_t\circ \alpha^{'}_t.
 \end{align}
 \end{definition}
 \begin{definition}
 A deformation $(\nu_t,\alpha_t)$ of $A$ is called trivial if $(\nu_t,\alpha_t)$ is equivalent to $(\nu_0,\alpha_0)$. A Hom-pre-Lie algebra $A$ is called rigid if it has only trivial deformation upto equivalence.
 \end{definition}
 
The Condition \ref{cond equiv defor} is equivalent to the following equations: 
\begin{align}
 &\sum_{i\geq 0}\psi_i\bigg( \sum_{j\geq 0}\nu\rq_j(a,b)t^j \bigg)t^i=\sum_{i\geq 0}\nu_i\bigg( \sum_{j\geq 0}\psi_j(a)t^j,\sum_{k\geq 0}\psi_k(b)t^k \bigg)t^i,\\
 &\sum_{i\geq 0}\alpha_i\bigg( \sum_{j\geq 0}\psi_j(a)t^j \bigg)t^i=\sum_{i\geq 0}\psi_i\bigg( \sum_{j\geq 0}\alpha\rq_j(a)t^j \bigg)t^i.
 \end{align}
 This is same as the following equations:
 \begin{align}
 \label{equivalent 10}&\sum_{i,j\geq 0}\psi_i(\nu\rq_j(a,b))t^{i+j}=\sum_{i,j,k\geq 0}\nu_i(\psi_j(a),\psi_k(b))t^{i+j+k},\\
 &\label{equivalent 101}\sum_{i,j \geq 0}\alpha_i(\psi_j(a))t^{i+j}=\sum_{i,j\geq 0}\psi_i(\alpha\rq_j(a))t^{i+j}.
 \end{align}
 Comparing constant terms on both sides of the above equations, we have
 \begin{align*}
 &\nu\rq_0(a,b)=\nu_0(a,b)=a\c b,\,\,\,\text{as}\,\,\,\psi_0=Id,\\
 &\alpha_0(a)=\alpha\rq_0(a)=\alpha(a).
  \end{align*}
  Now comparing coefficients of $t$, we have
  \begin{align}\label{equivalent main}
&\nu\rq_1(a,b)+\psi_1(\nu\rq_0(a,b))=\nu_1(a,b)+\nu_0(\psi_1(a),b)+\nu_0(a,\psi_1(b)),\\
\label{equivalent main 1}&\alpha_1(a)+\alpha_0(\psi_1(a))=\alpha\rq_1(a)+\psi_1(\alpha\rq_0(a)).
  \end{align}
 Equations (\ref{equivalent main}) and (\ref{equivalent main 1}) are same as
  \begin{align*}
& \nu\rq_1(a,b)-\nu_1(a,b)=a\c \psi_1(b)+\psi_1(a)\c b-\psi_1(a\c b)=\widetilde{\partial}_{\nu \nu} \psi_1(a,b).\\
& \alpha_1\rq (a) - \alpha_1(a) = \alpha(\psi_1(a)) - \psi_1(\alpha(a)) = \widetilde{\partial}_{\nu \alpha} \psi_1(a).
\end{align*}
Thus, we have the following proposition.
 \begin{proposition}
 Two equivalent formal deformations of $A$ have cohomologous infinitesimals.
 \end{proposition}

 \begin{theorem}
 Let $(\nu_t, \alpha_t)$ be a non-trivial formal deformation of the Hom-pre-Lie algebra $A$. Then $(\nu_t, \alpha_t)$ is equivalent to a deformation whose infinitesimal is not a coboundary.
 \end{theorem}
\begin{proof}
 Let $(\nu_t,\alpha_t)$ be a deformation of the Hom-pre-Lie algebra $A$ and $(\nu_n, \alpha_n)$ be the $n$-infinitesimal of the deformation for some $n\geq 1$. Then by Theorem (\ref{infinitesimal thm}), $(\nu_n, \alpha_n)$ is a $2$-cocycle, that is, $\widetilde{\partial} (\nu_n, \alpha_n)=0$. Suppose $\nu_n=-\widetilde{\partial}_{\nu \nu}\phi_n$ and $\alpha_n = -\widetilde{\partial}_{\nu \alpha}\phi_n$ for some $\phi_n\in \widetilde{C_\nu}^1(A, A)$, that is, $(\nu_n, \alpha_n)$ is a coboundary. We define a formal isomorphism $\Psi_t$ of $A[[t]]$ as follows:
  $$\Psi_t(a)=a+\phi_n(a)t^n.$$
  We set
  $$\bar{\nu_t}=\Psi^{-1}_t\circ \nu_t\circ (\Psi_t\otimes\Psi_t) \,\,\,\text{and}\,\,\,\bar{\alpha_t}=\Psi^{-1}_t\circ\alpha_t\circ\Psi_t.$$
  Thus, we have a new deformation $(\bar{\nu_t}, \bar{\alpha_t})$ which is equivalent to $(\nu_t, \alpha_t)$. By expanding the above equations and comparing coefficients of $t^n$, we get
  $$\bar{\nu_n}-\nu_n=\widetilde{\partial}_{\nu \nu} \phi_n,~~~\bar{\alpha_n} - \alpha = \widetilde{\partial}_{\nu \alpha}\phi_n.$$
  Hence, $\bar{\nu_n}=0, \bar{\alpha_n}=0$. By repeating this argument, we can kill off any infinitesimal which is a coboundary. Thus, the process must stop if the deformation is non-trivial. 
  \end{proof}
 
  \begin{cor}
 If $\widetilde{H^2}(A, A)=0$ then every deformation of $A$ is equivalent to a trivial deformation, that is, $A$ is rigid.
 \end{cor}

 \section{Equivariant cohomology of Hom-pre-Lie algebras}\label{equiv coho}\label{sec 4}
 In this section, we develop an equivariant cohomology for Hom-pre-Lie algebras equipped with a group action. To construct such an equivariant cohomology, we will follow Bredon's approach \cite{bredon67} for equivarint cohomology of topological $G$-spaces.
 
 \begin{definition}
 \emph{
 Let $G$ be a finite group and $A$ be a Hom-pre-Lie algebra. The group $G$ is said to act on $A$ from left if there exists a map
 $$\phi: G\times A \to A,~~~(g, a)\mapsto ga,$$
 satisfying
 \begin{enumerate}
 \item For each $g\in G$, the map $\phi(g,-)=\psi_g : A\to A,~a\mapsto ga$ is linear.
 \item For all $a\in A,~ea=a$, where $e$ denotes the identity element of $G$.
 \item $(g_1g_2)a= g_1(g_2 a)$ for all $g_1, g_2\in G$ and $a\in A$.
 \item $g(a \c b)=ga \c gb$ and $\alpha(g\c a)=g\alpha(a)$, for all $g\in G, a,b\in A.$
 \end{enumerate}
 }
 \end{definition}
 
 \begin{remark}
We defined finite group actions on Hom-pre-Lie algebras  for simplicity, but one can do it for any discrete groups.
 \end{remark}
 
Let $G$ is acting on the Hom-pre-Lie algebra $A$. Let $H$ be a subgroup of $G$. We define the $H$-fixed point set of $A$ as follows:
 $$A^H = \lbrace a\in A \mid ha=a, \forall h\in H\rbrace.$$

 \begin{lemma}
 For all subgroup $H$ of $G$, the $H$-fixed point set $A^H$ is a sub Hom-pre-Lie algebra of $A$.
 \end{lemma}
 \begin{proof}
 It is enough to show that $A^H$ is closed under multiplication and if $a\in A^H$ then $\alpha(a)\in A^H$.
 
 Let $a, b\in A^H$. Observe that
 $$h(a\c b) = ha\c hb = a\c b,~~~h(\alpha(a))= \alpha(h\c a)= \alpha(a),~\text{for all}~h\in H.$$
 Thus, $A^H$ is a sub Hom-pre-Lie algebra of $A$.
 \end{proof}

 An $O_G$-module is a contravariant functor
$$\mathcal{M}: O_G\to {\bf{Mod}}.$$
Note that ${\bf{Mod}}$ denotes a category of $\mathbb{K}$-modules, the category whose objects are $O_G$-modules with morphisms the natural transformations between $O_G$-modules is an abelian category denoted by $C_G$.

 \begin{definition}
An $O_G$-Hom-pre-Lie algebra is a contravariant functor
$$\mathcal{A}: O_G\to \bf{HpLie}.$$
$\bf{HpLie}$ denotes a category of Hom-pre-Lie algebras over $\mathbb{K}$ with Hom-pre-Lie algebra morphisms between them.
\end{definition}
\begin{example}
Suppose $A$ is a Hom-pre-Lie algebra equipped with an action of $G$, then we have a contravariant functor 
$$\chi_A:O_G\to\bf{HpLie},$$
given by $\chi_A(G/H)=A^H$ and for a morphism $\hat{g}:G/H\to G/K$ corresponding to the subconjugacy relation $g^{-1}Hg\subseteq K$,
$$\chi_A(\hat{g})=\psi_g: A^K\to A^H.$$
Thus, $\chi_A$ is an $O_G$-Hom-pre-Lie algebra.
\end{example}

A representation of an $O_G$-Hom-pre-Lie algebra $\mathcal{A}$ is an $O_G$-module $\mathcal{M}$ together with a natural transformation $\mathcal{\beta}$ on $\mathcal{M}$, and the following natural transformations
\begin{align*}
\mathcal{\rho}: \mathcal{A}\otimes \mathcal{M }\to\mathcal{M},~\mathcal{\mu} : \mathcal{M}\otimes \mathcal{A }\to\mathcal{M}
\end{align*}
such that for all $G/H\in\text{Ob}(O_G)$, $\mathcal{M}(G/H)$ is a non-equivariant representation of $\mathcal{A}(G/H)$ associated with natural transformations $\eta_l,\eta_r$.
\begin{remark}
From the above definition of the representation in an equivariant setting we have the following relations coming from the naturality of $\eta_l,\eta_r$.
\begin{align}
\mathcal{\rho}(G/H)\circ (\mathcal{A}(\hat{g})\otimes \mathcal{M}(\hat{g}))=\mathcal{M}(\hat{g})\circ \mathcal{\rho}(G/K),\\
\mathcal{\mu}(G/H)\circ (\mathcal{A}(\hat{g})\otimes \mathcal{M}(\hat{g}))=\mathcal{M}(\hat{g})\circ \mathcal{\mu}(G/K).
\end{align}

\end{remark}
\begin{remark}
It is well-known that any Hom-pre-Lie algebra $A$ is a representation over itself. In the equivariant setting also one can easily see that any $O_G$-Hom-pre-Lie algebra $\mathcal{A}$ can be thought of as a representation over itself. 
\end{remark}
We set $$S^n(\mathcal{A}; \mathcal{M}):=\bigoplus_{H\leq G}\widetilde{C^n}(\mathcal{A}(G/H), \mathcal{M}(G/H)).$$
We define
\begin{align*}
\delta^n:&S^n(\mathcal{A}; \mathcal{M})\to S^{n+1}(\mathcal{A};\mathcal{M}),\\
&\delta^n:=\bigoplus_{H\leq G}\widetilde{\partial}^n_H.
\end{align*}
Here $\widetilde{\partial}^H_n: \widetilde{C^n}(\mathcal{A}(G/H),\mathcal{M}(G/H))\to \widetilde{C^{n+1}}(\mathcal{A}(G/H), \mathcal{M}(G/H))$ is a non-equivariant coboundary map.

Clearly,$ \lbrace S^\sharp(\mathcal{A}; \mathcal{M}),\delta\rbrace$ is a cochain complex. Throughout this paper, we take $O_G$-Hom-pre-Lie algebra $\mathcal{A}$ as $\chi_A$ and consider the cohomology of $\chi_A$ with coefficients in $\chi_A$. We define a subcomplex of this cochain complex as follows:
\begin{definition}
A cochain $c=\lbrace (c^H_\nu, c^H_\alpha) \mid H\leq G\rbrace$ is said to be invariant if for every morphism $\hat{g}: G/H\to G/K$ corresponding to a subconjugacy relation $g^{-1}Hg\subseteq K$, following relation holds,
\begin{align*}
& c^H_\nu\circ \mathcal{A}(\hat{g})^{\otimes n} =\mathcal{A}(\hat{g})\circ c^K_\nu,\\
&  c^H_\alpha\circ \mathcal{A}(\hat{g})^{\otimes {n-1}} =\mathcal{A}(\hat{g})\circ c^H_\alpha.
\end{align*}
As $\chi_A$ is an $O_G$-Hom-pre-Lie algebra corresponding to $A$. This is same as,
\begin{align*}
& c^H_\nu\circ \psi_g^{\otimes n}=\psi_g\circ c^K_\nu,\\
& c^H_\alpha\circ \psi_g^{\otimes {n-1}}=\psi_g\circ c^K_\alpha.
\end{align*}
\end{definition}
\begin{lemma}
The set of all invariant $n$-cochains is a subgroup $S_G^n(\chi_A;\chi_A)$ of $S^n(\chi_A;\chi_A)$. If $c=\lbrace (c^H_\nu, c^H_\alpha) \rbrace\in S^n(\chi_A;\chi_A)$ is an invariant cochain then $\delta^n(c)=\lbrace \widetilde{\partial}^n_H(c^H_\nu, c^H_\alpha)\rbrace\in S^{n+1}(\chi_A;\chi_A)$ is an invariant $(n+1)$-cochain.
\end{lemma}
\begin{proof}
Let $c=\lbrace (c^H_\nu, c^H_\alpha) \rbrace\in S^n(\chi_A;\chi_A)$ is an invariant cochain. This implies
\begin{align*}
& c^H_\nu (ga_1,ga_2,\ldots, ga_n)=g c^K_\nu (a_1,a_2,\ldots,a_n),\\
& c^H_\alpha(ga_1,ga_2,\ldots, ga_{n-1})=gc^K_\alpha (a_1,a_2,\ldots,a_{n-1}).
\end{align*}
It is enough to show that four component differentials $\widetilde{\partial}_{\nu \nu}, \widetilde{\partial}_{\nu \alpha}, \widetilde{\partial}_{\alpha \alpha},~\text{and}~\widetilde{\partial}_{\alpha \nu}$ respect the group action. Note that 
\begin{align*}
&\widetilde{\partial}_{\nu \nu} (c^H_\nu)(\psi_g(a_1), \ldots, \psi_g(a_{n+1}))\\
&= \widetilde{\partial}_{\nu \nu} (c^H_\nu)(ga_1, \ldots, ga_{n+1})\\
&\sum^{n}_{i=1}(-1)^{i+1}\rho(\alpha^{n-1}(ga_i))c^H_\nu(ga_1,\ldots, \widehat{ga_i}, \ldots, ga_{n+1}) \\
&+ \sum^{n}_{i=1}(-1)^{i+1}\mu(\alpha^{n-1}(ga_{n+1}))c^H_\nu(ga_1,\ldots,\widehat{ga_i}, \ldots,  ga_{n}, ga_i)\nonumber\\
&-\sum^{n}_{i=1}(-1)^{i+1} c^H_\nu(\alpha(ga_1),\ldots, \widehat{ga_i}, \ldots,\alpha(g a_n), g a_i\c ga_{n+1})\nonumber\\
&=g\big(\sum^{n}_{i=1}(-1)^{i+1}\rho(\alpha^{n-1}(a_i))c^H_\nu(a_1,\ldots, \widehat{a_i}, \ldots, a_{n+1})\big) \\
&+ g\big(\sum^{n}_{i=1}(-1)^{i+1}\mu(\alpha^{n-1}(a_{n+1}))c^H_\nu(a_1,\ldots,\widehat{a_i}, \ldots,  a_{n}, a_i)\big)\nonumber\\
&- g\big(\sum^{n}_{i=1}(-1)^{i+1} c^H_\nu(\alpha(a_1),\ldots, \widehat{a_i}, \ldots,\alpha(a_n), a_i\c a_{n+1})\big)\nonumber\\
&= g \widetilde{\partial}_{\nu \nu} (c^H_\nu)(a_1, \ldots, a_{n+1})
\end{align*}
In a similar way, it is a routine work to check that other three components of the differential also preserves the group action. Thus, $\delta^n(c)=\lbrace \widetilde{\partial}^n_H(c^H_\nu, c^H_\alpha)\rbrace\in S^{n+1}(\chi_A;\chi_A)$ is an invariant $(n+1)$-cochain. 
\end{proof}

Thus, $(S^{\sharp}_G(\chi_A;\chi_A), \delta)$ is a cochain complex and homology of this complex is called equivariant cohomology of the Hom-pre-Lie algebra $A$. We denote $n$th cohomology of $A$ by $\widetilde{H}_G^n(A, A)$.
 
% Let $A$ be a Hom-pre-Lie algebra equipped with an action of $G$. \\
% Set
% \begin{align*}
% &\widehat{C_G^n}(A,~A)\\
% &=\lbrace (c_{\nu}, c_{\alpha})\in \widehat{C^n}(A,~A)  : c_\nu(\psi_g(a_1),\ldots,\psi_g(a_n))=gc_\nu(a_1,\ldots,a_n),\\
%&~ c_\alpha(\psi_g(a_1),\ldots,\psi_g(a_{n-1}))=gc_\alpha(a_1,\ldots,a_{n-1})\rbrace\\
%&=\lbrace (c_\nu, c_\alpha) \in \widetilde{CL}^n(L, L) : c_\nu(ga_1,\ldots,ga_n)=gc_\nu(a_1,\ldots,a_n),\\
%&~ c_\alpha(ga_1,\ldots,ga_{n-1})=gc_\alpha(a_1,\ldots,a_{n-1}) \rbrace.
%\end{align*}
%Note that $\widehat{C_G^n}(A,~A)$ consist od all cochains of $\widehat{C^n}(A,~A)$ which are equivariant, that is, those preserves the group action.

%\begin{lemma}
%If $(c_{\nu}, c_{\alpha})\in \widehat{C_G^n}(A,~A)$ then $\partial(c_{\nu}, c_{\alpha})\in \widehat{C_G^{n+1}}(A,~A)$.
%\end{lemma}
%\begin{proof}

%\end{proof}

\section{Equivariant formal deformation of Hom-pre-Lie algebras}\label{sec 5}
In this section, we study a one-parameter formal deformation theory for Hom-pre-Lie algebras equipped with a group action. We also show that equivariant cohomology defined in Section \ref{equiv coho} is the deformation  cohomology of such formal deformations.

Let $\chi_A$ is an $O_G$-Hom-pre-Lie algebra. For any subgroup $H$ of $G$, we denote $A^H[[t]]$ as a formal power series ring in variable $t$ over $H$-fixed point sets $A^H$.
 \begin{definition} \label{equi formal deform defn}
 A one-parameter formal deformation of $\chi_A$ is a family consisting of pair of natural transformation $(\mathfrak{v}_t, \mathfrak{a}_t)$ such that for every $G/H\in \text{Ob}(O_G)$, components of $\mathfrak{v}_t$ are bilinear maps, and components of $\mathfrak{a}_t$ are (invertible) linear maps
\begin{align*}
&  \mathfrak{v}_t (G/H): \chi_A(G/H)[[t]]\times \chi_A(G/H)[[t]]\to \chi_A(G/H)[[t]],\\
&\mathfrak{a}_t(G/H): \chi_A(G/H)[[t]]\to \chi_A(G/H)[[t]].
 \end{align*} 
 This is same as
 \begin{align*}
 \nu^{ H}_t : A^H[[t]]\times A^H[[t]]\to A^H[[t]],~~~\alpha^{H}_t : A^H[[t]]\to A^H[[t]]
 \end{align*}
 which can be expressed in the following form
 \begin{align*}
 &\nu^{H}_t(a,b)= \nu^{H}_0(a,b)+\nu^{H}_1(a,b)t+\nu^{H}_2(a,b)t^2+\cdots,\\
 &\alpha^{H,}_t(a)= \alpha^{H}_0(a)+ \alpha_1^{H}(a)t+\alpha^{H}_2(a)t^2+\cdots,
 \end{align*}
 such that
 \begin{enumerate}
 \item $\nu^{H}_0(a,b)=a\c b$ and $\alpha^{H}_0(a)=\alpha(a)$ for $a,b\in A$.
\item For $i\geq 0$, $\nu^{H}_i$ and $\alpha^{H}_i$ are  $\mathbb{K}$-bilinear and  $\mathbb{K}$-linear maps respectively on $A^H$, and satisfies the following Hom-pre-Lie algebra relations.
 $$\nu^H_t(\nu^H_t(a,b), \alpha^H_t(c)) - \nu^H_t (\alpha_t(a), \nu^H_t(b,c)) = \nu^H_t(\nu^H_t(b,a), \alpha^H_t(c)) - \nu^H_t (\alpha^H_t(b), \nu^H_t(a,c)),$$ for all $a,b,c\in A^H$ and $H\leq G$. \label{equi deform cond 2} 
 
 \item For all the subgroups $H$ of $G$, the map $\alpha^H_t$ is multiplicative, that is, $\nu^H_t (\alpha^H_t(a), \alpha^H_t(b)) = \alpha^H_t (\nu^H_t(a,b))$.\label{equi deform cond 3}
 
 \item For every morphism $\hat{g}:G/H\to G/K$ and $g\in G$ satisfying $g^{-1}Hg\subseteq K$ the following relations hold-
 \begin{align*}
& \nu^{H}_t\circ (\chi_A(\hat{g})\otimes \chi_A(\hat{g}))=\chi_A(\hat{g})\circ \nu^{K}_t,\\
&\alpha_t^H \circ \chi_A(\hat{g}) = \alpha_t^K\circ \chi_A(\hat{g}).
  \end{align*}
 \end{enumerate}
 \end{definition}
 
 From the Condition \ref{equi deform cond 2}, we have the following equations for all $n\geq 0$.
 \begin{align}\label{equi deform hom-pre-lie 2}
 \sum_{\substack{i+j+k = n\\i,j,k\geq 0}} \nu^H_i(\nu^H_j(a,b), \alpha^H_k(c)) - \nu^H_i (\alpha^H_j(a), \nu^H_k(b,c)) - \nu^H_i(\nu^H_j(b,a), \alpha^H_k(c)) + \nu^H_i (\alpha^H_j(b), \nu^H_k(a,c)) = 0.
 \end{align}
 The Condition \ref{equi deform cond 3} is equivalent to the following equation:
 \begin{align} \label{equi deform alpha}
 \sum_{\substack{i+j+k = n\\i,j,k\geq 0}} \nu^H_i (\alpha^H_j(a), \alpha^H_k(b)) - \sum_{\substack{i+j=n\\i,j\geq 0}} \alpha^H_i (\nu^H_j(a,b))=0.
 \end{align}
 \begin{align}
 &  (\widetilde{\partial}_{\nu\nu} \nu^H_n - \widetilde{\partial}_{\alpha\nu}\alpha^H_n)(a,b,b)\\
 &= \sum_{\substack{i+j+k = n\\i,j,k\geq 0}} \nu^H_i(\nu^H_j(a,b), \alpha^H_k(c)) - \nu^H_i (\alpha^H_j(a), \nu^H_k(b,c)) - \nu^H_i(\nu^H_j(b,a), \alpha^H_k(c)) + \nu^H_i (\alpha^H_j(b), \nu^H_k(a,c)).
 \end{align}
 Note that
 \begin{align*}
 (\widetilde{\partial}_{\nu\nu} \nu^H_n)(a, b, c) = &\alpha(a)\c \nu^H_n(b,c) - \alpha(b)\c \nu^H_n(a,c) + \nu^H_n(b,a)\c \alpha(c) - \nu^H_n(a, b)\c \alpha(c) - \nu^H_n(\alpha(b), a\c c)  \\
& + \nu^H_n(\alpha(a), b\c c) - \nu^H_n(a\c b, \alpha(c)) - \nu^H_n(b\c a, \alpha(c)),~\text{for all}~a,b,c\in A^H.
 \end{align*}
 and
 \begin{align*}
 (\widetilde{\partial}_{\alpha^H\nu^H}\alpha^H_n)(a, b, c) = \alpha^H_n(b) \c (a \c c) - \alpha^H_n(a) (b\c c)+ (a \c b) \c \alpha^H_n(c) - (b\c a)\c \alpha^H_n(c)
 \end{align*}
 The Equation \ref{equi deform alpha} is same as the following equation involving differentials.
 \begin{align}\label{equi deform alpha diff equ}
  (\widetilde{\partial}_{\nu\alpha}\nu^H_n-\widetilde{\partial}_{\alpha \alpha} \alpha^H_n)(a, b) = \sum_{\substack{i+j+k = n\\0\leq i,j,k\leq n-1}} \nu^H_i (\alpha^H_j(a), \alpha^H_k(b)) - \sum_{\substack{i+j=n\\ 0\leq i,j\leq n-1}} \alpha^H_i (\nu^H_j(a,b)).
 \end{align}
 Observe that
 \begin{align*}
 (\widetilde{\partial}_{\alpha\alpha} \alpha^H_n) (a,b) = \alpha(a)\c \alpha^H_n(b) + \alpha^H_n(a)\c \alpha(b) - \alpha^H_n(a\c b).
 \end{align*}
 and 
  \begin{align*}
 (\tilde{\partial_{\nu\alpha}} \nu^H_n) (a,b) = \alpha(\nu^H_n(a,b)) -\nu^H_n(\alpha(a), \alpha(b)).
 \end{align*}
 
 \begin{definition}
 We define
 $$\nu^G_1 = \oplus_{H\leq G} \nu_1^H,~~~\alpha^G_1 =\oplus_{H\leq G} \alpha^H_1$$
 The pair $(\nu^G_1,\alpha^G_1)$ is called the equivariant infinitesimal of the given deformation $(\mathfrak{v}_t, \mathfrak{a}_t)$. More generally, assume that $(\nu^H_n, \alpha^H_n)$ is the first non-zero term after $(\nu^H_0,\alpha^H_0)$ of $(\mathfrak{v}_t, \mathfrak{a}_t)$, for all subgroups $H$ of $G$ and for some $n\geq 0$. We define the $n$th equivariant infinitesimal of the equivariant deformation $(\mathfrak{v}_t, \mathfrak{a}_t)$ as the pair $(\nu^G_n, \alpha^G_n)$, where
 $$\nu^G_n = \oplus_{H\leq G} \nu_n^H,~~~\alpha^G_n =\oplus_{H\leq G} \alpha^H_n$$
 \end{definition}
 Similar to the non-equivariant case, we have the following theorem.
 \begin{theorem}
 Let $(\mathfrak{v}_t, \mathfrak{a}_t)$ be an equivariant one-parameter formal deformation of the Hom-pre-Lie algebra $A$. Then the equivariant infinitesimal of a given deformation of $A$ equipped with an action of $G$ is a $2$-cocycle of the equivariant $\alpha$-type cohomology of $A$.
 \end{theorem}
 
Now, we show how equivariant $\alpha$-type cohomology control obstructions of given equivariant formal deformations.
 \begin{definition}
 \emph{
 An equivariant deformation of order $n$ of $A$ eqipped with action of $G$ is given by a $\mathbb{K}[[t]]$-bilinear map $\nu^H_t : A^H[[t]]\times A^H[[t]] \to A^H[[t]]$ and a $\mathbb{K}[[t]]$-linear map $\alpha^H_t : A^H[[t]] \to A^H[[t]]$ of the forms
 $$\nu^H_t = \sum^{n}_{i=0} \nu^H_i t^i,~~~\alpha^H_t = \sum^n_{i=0} \alpha^H_i t^i,$$
 such that the pair $(\nu^H_t, \alpha^H_t)$ satisfies the definition of the formal deformation of $A^H$ (mod $t^{n+1}$), for all subgroups $H$ of $G$.
 }
 \end{definition}
 We say an equivariant deformation $(\mathfrak{v}_t, \mathfrak{a}_t)$ of order $n$ is extendable to a deformation of order $(n+1)$ if for all subgroups $H$, there exists $\nu^H_{n+1} \in \widetilde{C_\nu^{n+1}}(A,A)$ and $\alpha^H_{n+1} \in \widetilde{C_\alpha^{n+1}}(A,A)$ such that 
 $$\bar{\nu^H_t} = \nu^H_t +\nu^H_{n+1}t^{n+1},~~~\bar{\alpha^H_t} = \alpha^H_t +\alpha^H_{n+1}t^{n+1},$$
 and the pair $(\bar{\nu^H_t},\bar{\alpha^H_t})$ satisfies the Definition \ref{equi formal deform defn} mod($t^{n+2}$).
  From an equivariant deformation $(\mathfrak{v}_t, \mathfrak{a}_t)$ of order $(n+1)$, we have the following equations:
   \begin{align*}
& \sum_{\substack{i+j+k = n+1\\i,j,k\geq 0}} \nu^H_i(\nu^H_j(a,b), \alpha^H_k(c)) - \nu^H_i (\alpha^H_j(a), \nu^H_k(b,c)) - \nu^H_i(\nu^H_j(b,a), \alpha^H_k(c)) + \nu^H_i (\alpha^H_j(b), \nu^H_k(a,c)) = 0,\\
& \sum_{\substack{i+j+k = n+1\\i,j,k\geq 0}} \nu^H_i (\alpha^H_j(a), \alpha^H_k(b)) - \sum_{\substack{i+j=n+1\\i,j\geq 0}} \alpha^H_i (\nu^H_j(a,b))=0.
 \end{align*}
 We can re-write the above equations as follows:
 \begin{align*}
& (\widetilde{\partial}_{\nu\nu} \nu^H_n - \widetilde{\partial}_{\alpha\nu}\alpha^H_n)(a, b, c) =  \sum_{\substack{i+j+k=n+1\\0\leq  i,j,k\leq n}} ( \nu^H_i \circ_{\alpha^H_j} \nu^H_k) (a, b, c) -   \sum_{\substack{i+j+k=n+1\\ 0\leq i,j,k\leq n}} (\nu^H_i \circ_{\alpha^H_j} \nu^H_k) (b, a, c),\\ 
& (\widetilde{\partial}_{\alpha\alpha} \alpha^H_n - \widetilde{\partial}_{\nu\alpha}\nu^H_n)(a, b) = \sum_{\substack{i+j+k = n+1\\0\leq i,j,k\leq n}} \nu^H_i (\alpha^H_j(a), \alpha^H_k(b)) - \sum_{\substack{i+j=n+1\\ 0\leq i,j\leq n}} \alpha^H_i (\nu^H_j(a,b)).
 \end{align*}
 We define the $n$th equivariant obstruction to extend an equivariant deformation of order $n$ to $n+1$ as $O_G^n(A):= (\oplus_{H\leq G}O^{n,H}_\nu (A), \oplus_{H\leq G}O^{n,H}_\alpha (A))$, where
 \begin{align}
& O^{n,H}_\nu (A)= \sum_{\substack{i+j+k=n+1\\0\leq  i,j,k\leq n}} ( \nu^H_i \circ_{\alpha^H_k} \nu^H_j) (a, b, c) -   \sum_{\substack{i+j+k=n+1\\ 0\leq i,j,k\leq n}} (\nu^H_i \circ_{\alpha^H_k} \nu^H_j) (b, a, c),\\
& O^{n,H}_\alpha (A)= \sum_{\substack{i+j+k = n+1\\0\leq i,j,k\leq n}} \nu^H_i (\alpha^H_j(a), \alpha^H_k(b)) - \sum_{\substack{i+j=n+1\\ 0\leq i,j\leq n}} \alpha^H_i (\nu^H_j(a,b)).
 \end{align}
 The proofs of the following results are similar to the proofs of the corresponding non-equivariant results.
 \begin{theorem}
 An equivariant deformation $(\mathfrak{v}_t, \mathfrak{a}_t)$ of order $n$ is extendable to a deformation of order $n+1$ if and only if the cohomology class of $O_G^n(A)$ vanishes.
 \end{theorem} 
 
 \begin{cor}
 The obstruction of extending an equivariant infinitesimal deformation lies in $\widetilde{H^3}_G(A, A)$.
 \end{cor}

 Let $(\mathfrak{v}_t, \mathfrak{a}_t)$ and $(\mathfrak{v}^{\rq}_t, \mathfrak{a}^{\rq}_t)$ be two equivariant formal deformations of $A$, where 
 $(\nu^{'})^H_t = \sum_{i\geq 0} \nu^{'}_i t^i,~~~(\alpha^{'})^H_t = \sum_{i\geq 0} \alpha^{'}_i t^i$, for all subgroups $H$ of $G$.
 \begin{definition}
 Two equivariant deformations $(\mathfrak{v}_t, \mathfrak{a}_t)$ and $(\mathfrak{v}^{\rq}_t, \mathfrak{a}^{\rq}_t)$  are said to be equivalent if for all subgroups $H$ of $G$, there exists $\mathbb{K}[[t]]$-isomorphism $\Psi^H_t : A[[t]] \to A[[t]]$ of the form $\Psi^H_t = \sum_{i\geq 0} \psi^H_i t^i$ such that for all $i\geq 0$, 
 \begin{align}
 \label{equi cond equiv defor}
 & 1.~~~ \psi^H_0= \text{id}\\
 &2.~~~ \Psi^H_t \circ (\nu^{'}_t)^H = \nu^H_t \circ (\Psi^H_t \otimes \Psi^H_t),~\text{and}~ \alpha^H_t\circ \Psi^H_t = \Psi^H_t\circ (\alpha^{'}_t)^H.
 \end{align}
 \end{definition}
 \begin{definition}
 An equivariant deformation $(\mathfrak{v}_t, \mathfrak{a}_t)$ of $A$ is called trivial if $(\mathfrak{v}_t, \mathfrak{a}_t)$ is equivalent to $(\mathfrak{v}_0, \mathfrak{a}_0)$. A Hom-pre-Lie algebra $A$ equipped with a group action is called rigid if it has only trivial deformation upto equivalence.
 \end{definition}
 
 This is same as the following equations:
 \begin{align}
 \label{equivalent 10}&\sum_{i,j\geq 0}\psi^H_i((\nu\rq_j)^H(a,b))t^{i+j}=\sum_{i,j,k\geq 0}\nu_i(\psi^H_j(a),\psi^H_k(b))t^{i+j+k},\\
 &\label{equivalent 101}\sum_{i,j \geq 0}\alpha^H_i(\psi^H_j(x))t^{i+j}=\sum_{i,j\geq 0}\psi^H_i((\alpha\rq_j)^H(x))t^{i+j}.
 \end{align}
 
  Now comparing coefficients of $t$, we have
  \begin{align}\label{equi equivalent main}
&(\nu\rq_1)^H(a,b)+\psi^H_1((\nu\rq_0)^H(a,b))=\nu^H_1(a,b)+\nu^H_0(\psi^H_1(a),b)+\nu^H_0(a,\psi^H_1(b)),\\
\label{equi equivalent main 1}&\alpha^H_1(a)+\alpha^H_0(\psi^H_1(a))=(\alpha\rq_1)^H(a)+\psi^H_1((\alpha\rq_0)^H(a)).
  \end{align}
  The Equations (\ref{equi equivalent main}) and (\ref{equi equivalent main 1}) are same as
  \begin{align*}
& (\nu\rq_1)^H(a,b)-\nu^H_1(a,b)=a\c \psi^H_1(b)+\psi^H_1(a)\c b-\psi^H_1(a\c b)=\widetilde{\partial}_{\nu\nu} \psi^H_1(a,b).\\
& (\alpha_1\rq)^H (a) - \alpha^H_1(a) = \alpha(\psi^H_1(a)) - \psi^H_1(\alpha(a)) = \widetilde{\partial}_{\nu \alpha} \psi^H_1(a).
\end{align*}
Similar to the non-equivariant case, one can show the following results.
 \begin{proposition}
 Two equivalent formal deformations of $A$ equipped with an action of $G$ have cohomologous infinitesimals.
 \end{proposition}

 \begin{theorem}
 Let $(\mathfrak{v}_t, \mathfrak{a}_t)$ be a non-trivial equivariant formal deformation of the Hom-pre-Lie algebra $A$. Then $(\mathfrak{v}_t, \mathfrak{a}_t)$ is equivalent to a deformation whose infinitesimal is not a coboundary.
 \end{theorem}

  \begin{cor}
 If $\widetilde{H^2}_G(A, A)=0$ then every equivariant deformation of $A$ is equivalent to a trivial deformation, that is, $A$ is equivariantly rigid.
 \end{cor}

%\begin{definition}
%An equivariant one-parameter formal deformation of a Hom-pre-Lie algebra $A$ equipped with an action of $G$ is given by $\mathbb{K}[[t]]$-bilinear and a $\mathbb{K}[[t]]$-linear map $\nu_t:A[[t]]\times A[[t]]\to A[[t]]$ and $\alpha_t:A[[t]]\to A[[t]]$ respectively of the form
%$$\nu_t=\sum_{i\geq 0}\nu_it^i\,\,\text{and}\,\,\alpha_t=\sum_{i\geq 0}\alpha_it^i,$$
%where each $\nu_i:A\times A\to A$ is a $\mathbb{K}$-bilinear map and each $\alpha_i: A\to A$ is a $\mathbb{K}$-linear map satisfying the followings:
%\begin{enumerate}
%\item $m_0(a,b)=a.b$ is the original multiplication operation of $A$ and $\alpha_0(a,b)=\alpha(a,b)$.
%\item $\nu_t$ and $\alpha_t$ satisfies the following Hom-pre-Lie algebra identity:\label{equ defor}
%$$\nu_t(\nu_t(a,b), \alpha_t(c)) - \nu_t (\alpha_t(a), \nu_t(b,c)) = \nu_t(\nu_t(b,a), \alpha_t(c)) - \nu_t (\alpha_t(b), \nu_t(a,c)).$$
%\item The map $\alpha_t$ is multiplicative, that is, $m_t (\alpha_t(a), \alpha_t(b)) = \alpha_t (m_t(a, b))$.\label{equ defor mult}
%\item For all $g\in G$,~$a,b\in A$ and $i\geq 0$,$$\nu_i(ga,gb)=g\nu_i(a,b)\,\,\text{and}\,\,\alpha_i(ga)=g\alpha_i(a).$$
%That is, $m_i\in\text{Hom}^G_\mathbb{K}(A\otimes A, A)$ and $\alpha_i\in\text{Hom}^G_\mathbb{K}(A,A).$
%\end{enumerate}
%\end{definition}
 
\def\theequation{\arabic{section}. \arabic{equation}}
\setcounter{equation} {0}

\begin{center}
 {\bf ACKNOWLEDGEMENT}
 \end{center}
The first author is supported by the NSF of China (No. 12161013)  and Guizhou Provincial Science and Technology Foundation (No.ZK [2023]025). The second/corresponding author is supported by the Science and Engineering Research Board (SERB), Department of Science and Technology (DST), Govt. of India. (Grant Number- CRG/2022/005332). The authors would like to thank esteemed referees for carefully reading the manuscript and giving such constructive suggestions and comments which substantially helped improving the quality of the paper.

\renewcommand{\refname}{REFERENCES}


\begin{thebibliography}{99}
\bibitem{MK03}
B. Bakalov and V. Kac, Field algebras, \emph{Int. Math. Res. Not.} \textbf{3} (2003) 123-159.

\bibitem{B97} D. Balavoine, Deformation of algebras over a quadratic operad. Contemp. Math.  \textbf{202} (1997) 207-234.

%\bibitem{B60} G. Baxter, An analytic problem whose solution follows from a simple algebraic identity, Pacific J. Math. 10 (1960) 731-742.

%\bibitem{BBGN13} C. Bai, O. Bellier, L. Guo and X. Ni, Splitting of operations, Manin products, and
%Rota-Baxter operators, Int. Math. Res. Not. 3 (2013) 485-524.

\bibitem{bredon67}
G. E. Bredon, Equivariant cohomology theories, \emph{Lecture Notes in Math.}, No. \textbf{34}, Springer-Verlag, Berlin, 1967.

\bibitem{CL01}
F. Chapoton and M. Livernet, Pre-Lie algebras and the rooted trees operad, \emph{Int. Math. Res. Not.} \textbf{8} (2001) 395-408.

\bibitem{CK00} A. Connes and D. Kreimer, Renormalization in quantum field theory and the Riemann-Hilbert problem. I. The Hopf algebra structure of graphs and the main theorem, \emph{Comm. Math. Phys.} \textbf{210} (2000) 249-273.

%\bibitem{CS11}  Y.  Cheng and  Y. Su, (Co)homology and universal central extension of
%Hom-Leibniz algebras, Acta Math. Sin. (Engl. Ser.) 27 (2011),  813-830.


%\bibitem{Da20}  A. Das, Deformations of associative Rota-Baxter operators, J. Algebra 560 (2020) 144-180.


%\bibitem{FZ15} Y. Fr\'{e}gier and M. Zambon, Simultaneous deformations of algebras and morphisms
%via derived brackets, J. Pure Appl. Algebra 219 (2015) 5344-5362.

\bibitem{G63} M. Gerstenhaber, The cohomology structure of an associative ring, \emph{Ann. Math.} \textbf{78}
(1963) 267-288.
\bibitem{G64} M. Gerstenhaber, On the deformation of rings and algebras, \emph{Ann. Math.} \textbf{(2)} 79
(1964) 59-103.
%\bibitem{G12} L. Guo, An Introduction to Rota-Baxter Algebra (Higher Education Press, Beijing,
%2012).


%\bibitem{HLS06}  J. Hartwig, D. Larson and S. Silvestrov, Deformations of Lie algebras using
%$\sigma$-derivations, J. Algebra 295 (2006) 314-361.

%\bibitem{Hu} N.  Hu,
%q-Witt algebras, q-Lie algebras, q-holomorph structure and representations,
% Algebr. Colloq. (6)1999,  51-70.
 
 \bibitem{HM19}
B. Hurle and A. Makhlouf, $\alpha$-type Hochschild cohomology of Hom-associative algebras and bialgebras. \emph{J. Korean Math. Soc.} \textbf{56} (2019), no.6, 1655-1687.

\bibitem{HM19glas}
B. Hurle and A. Makhlouf, $\alpha$-type Chevalley-Eilenberg cohomology of Hom-Lie algebras and bialgebras, \emph{Glasgow Mathematical Journal}, Volume \textbf{62}, Special Issue S1 - December 2020.

\bibitem{LM88}
A. Lichnerowicz and A. Medina, On Lie groups with left-invariant symplectic or Kahlerian structures, \emph{Lett. Math. Phys.} \textbf{16} (1988) 225-235.

\bibitem{LMS22}
S. Liu, A. Makhlouf and L. Song, The full cohomology, abelian
extensions and formal deformations of Hom-pre-Lie algebras, \emph{Electronic Research Archive}, Vol. \textbf{30} Issue 8, (2022) 2748-2773.

\bibitem{LST20} S. Liu, L. Song and R. Tang,  Representations and cohomologies of regular
Hom-pre-Lie algebras, \emph{J. Algebra Appl.}  \textbf{19}(2020),  2050149  (22 pages).

%\bibitem{L93} J. -L. Loday, Une version non commutative des alg\`{e}bres de Lie: les alg\`{e}bres de Leibniz,  Enseign. Math. (2) 39(1993) 269-293.
%\bibitem{LP93}  J. -L. Loday and T. Pirashvili, Universal enveloping algebras of Leibniz algebras and (co)homology,  Math. Ann.
%296 (1993) 139-158.

\bibitem{M10}  A. Makhlouf,  Hom-alternative algebras and Hom-Jordan algebras, \emph{Int. Electron. J. Algebra} \textbf{8} (2010), 177-190.

\bibitem{M08}  A. Makhlouf and S. Silvestrov,  Hom-algebra structures, \emph{J. Gen. Lie Theory Appl.}, \textbf{2} (2008), 51-64.

\bibitem{MS10} A. Makhlouf and S. Silvestrov,  Hom-algebras and Hom-coalgebras, J. Algebra Appl. 9(2010),  553-589.
\bibitem{MZ18}  A. Makhlouf and P. Zusmanovich, Hom-Lie structures on Kac-Moody algebras, \emph{Journal of Algebra} \textbf{515}(2018), 278-297.

\bibitem{MS20} G. Mukherjee and R. Saha,  Equivariant one-parameter formal deformations of Hom-Leibniz algebras,  \emph{Communications in Contemporary Mathematics} (2020), Vol. \textbf{24}, No. 03, 2050082 (2022).

\bibitem{NR66} A. Nijenhuis and R. Richardson, Cohomology and deformations in graded Lie algebras, \emph{Bull. Amer. Math. Soc.} \textbf{72} (1966) 1-29.

\bibitem{NR68} A. Nijenhuis and R. Richardson, Commutative algebra cohomology and deformations
of Lie and associative algebras, \emph{J. Algebra} \textbf{9} (1968) 42-105.

%\bibitem{PBG17} J. Pei, C. Bai and L. Guo, Splitting of operads and Rota-Baxter operators on operads,
%Appl. Categ. Structures 25 (2017) 505-538.

\bibitem{saha19} 
R. Saha, Equivariant associative dialgebras and its one-parameter formal deformations, \emph{Journal of Geometry and Physics}, volume 146, 2019, 103491.

\bibitem{S20} R. Saha (2020), Cup-product in Hom-Leibniz cohomology and Hom-Zinbiel algebras, \emph{Communications in Algebra}, \textbf{48}:10, 4224-4234, \url{DOI: 10.1080/00927872.2020.1759613}.

\bibitem{SL17}
Q. Sun and H. Li, On parak\"ahler Hom-Lie algebras and hom-left-symmetric bialgebras, \emph{Communications in Algebra} \textbf{45}(1) (2017) 105-120.

\bibitem{V63}
E. B. Vinberg, The theory of homogeneous convex cones, \emph{Transl. Moscow Math. Soc.} \textbf{12} (1963), 340-403.


\end{thebibliography}
\end{document}